\documentclass[12pt]{article}
\usepackage{latexsym}
\usepackage{tikz}
\usetikzlibrary{arrows,shapes,positioning}
\usetikzlibrary{decorations.markings}
\tikzstyle arrowstyle=[scale=2]
\tikzstyle directed=[postaction={decorate,decoration={markings,
    mark=at position .65 with {\arrow[arrowstyle]{stealth}}}}]

\tikzstyle{vertex}=[circle, draw, fill=black, inner sep=0pt, minimum size=4pt]
\tikzstyle{dot}=[circle, draw, fill=black, inner sep=0pt, minimum size=2pt] 

\usepackage{amssymb} 
\usepackage{amsmath} 

\newtheorem{theorem}{Theorem}[section]
\newtheorem{lemma}[theorem]{Lemma}
\newtheorem{proposition}[theorem]{Proposition}
\newtheorem{prop}[theorem]{Proposition}
\newtheorem{definition}[theorem]{Definition}
\newtheorem{corollary}[theorem]{Corollary}

\newcommand{\qed}{\ \hfill\mbox{$\Box$}\vspace{\baselineskip}}
\newenvironment{proof}{\noindent {\bf Proof:}}{{\qed}}

\DeclareMathOperator\link{link}

\definecolor{purple}{rgb}{.5, 0, .635} 

\newcommand\abs[1]{\left| #1 \right|}

\usepackage{ifxetex,ifluatex}
\newif\ifxetexorluatex
\ifxetex
  \xetexorluatextrue
\else
  \ifluatex
    \xetexorluatextrue
  \else
    \xetexorluatexfalse
  \fi
\fi

\ifxetexorluatex
  \usepackage{fontspec}
\else
  \usepackage[T1]{fontenc}
  \usepackage[utf8]{inputenc}
  \usepackage{lmodern}
\fi

\usepackage[hidelinks]{hyperref}

\title{Manifold Matching Complexes}
\author{Margaret Bayer,
Bennet Goeckner and Marija Jeli\'{c} Milutinovi\'{c}}

\begin{document}

\maketitle

\begin{abstract}
The matching complex of a graph is the simplicial complex whose vertex set is
the set of edges of the graph with a face for each independent set of edges.
In this paper we completely characterize the pairs (graph, matching complex) for which
the matching complex is a homology manifold, with or without boundary. 
Except in dimension two, all of these manifolds are spheres or balls. 
\end{abstract}

\section{Introduction}
The matching complex $M(G)$ of a graph $G$ is a simplicial complex representing the matchings (sets of independent edges) of the graph.
There is an extensive literature describing the matching complexes of certain types of graphs.

There are many results on the topology of the matching complexes of
interesting classes of graphs.
For example, there has been much study of chessboard complexes,
$\Delta_{m,n}=M(K_{m,n})$.
Bj\"{o}rner, et al.,~\cite{blvz}  prove that $M(K_{m,n})$ is
$\nu$-connected,
where $\nu=\min\{m,n,\lfloor\frac{m+n+1}{3}\rfloor\}-2$. 
Ziegler~\cite{ziegler} shows that for $m\ge 2n-1$, $M(K_{m,n})$ is shellable,
and Joji\'{c}~\cite{jojic} uses this to give a recursion for the $h$-vectors
of these chessboard complexes.
Athanasiadis~\cite{Athan} studies vertex decomposability of skeleta of
hypergraph matching complexes and chessboard complexes.
Wachs~\cite{Michelle} surveys results on the homology of chessboard complexes
and matching complexes of the complete graph.
Jonsson's dissertation (published as \cite{Jakob}) studies various complexes
associated with graphs, including matching complexes.
Kozlov \cite{Kozlov} proves that, for
$\nu_n=\lfloor \frac{n-2}{3}\rfloor$,
 the matching complex $M(P_{n+1})$ of
the length $n$ path is homotopy equivalent to the sphere $S^{\nu_n}$ when
$n\bmod{3}\ne 1$, and
 the matching complex $M(C_n)$ of
the $n$-cycle is homotopy equivalent to the sphere $S^{\nu_n}$ when
$n\bmod{3}\ne 0$.
(As is standard in graph theory, the subscript on
a graph name indicates the number of vertices; for paths this is one
more than the length.)
Matching complexes of grid graphs have been studied by Braun and Hough 
\cite{Braun_Hough} and by Matsushita \cite{Matsushita}.
Marietti and Testa \cite{MT08} proved that matching complexes of forests
are contractible or homotopy equivalent to a wedge of spheres. For
caterpillar graphs, Jelić Milutinović, et al. \cite{JJMV} give explicit formulas
for the number of spheres in each dimension. They also study the
connectivity of matching complexes of honeycomb graphs.

We are interested in the reverse question: which simplicial complexes are matching complexes of graphs?
In this paper we will classify homology manifolds, with and without
boundary, that are matching complexes. 
In Section~\ref{Prelim Section}, we
review definitions and introduce several tools that we will rely on in later sections.
In Section~\ref{Sphere Section}, we describe all graphs whose matching complexes
are 1- and 2-dimensional spheres.
In Section~\ref{Manifolds_w/o_boundary Section},
we describe all homology manifolds
without boundary that arise as matching complexes.  All of these matching complexes are combinatorial spheres, except in dimension two, where the torus is also a matching complex.
In
Section~\ref{Manifolds_w/_Boundary Section}, we finish the story with
a complete
description of the matching complexes that are homology
manifolds with boundary.
In dimension two, a variety of manifolds with boundary
arise as matching complexes.
In dimensions three and higher, these matching complexes are all  combinatorial balls.
Moreover, the graphs that produce manifold
matching complexes are all constructed from the disjoint union of copies of
a finite set of graphs, which we explicitly specify.

\section{Preliminaries}\label{Prelim Section}
\subsection{General properties of matching complexes}

\begin{definition}
{\em
A {\em matching} of a graph $G$ is a set of edges of $G$, no two of which share a vertex.
}\end{definition}

\begin{definition}{\em
The {\em matching complex} of a graph $G$ is the simplicial complex $M(G)$ with vertex set the set $E$ of edges of $G$ and simplices every subset $\sigma\subseteq E$ that forms a matching of $G$.
}\end{definition}

In what follows we will use the notational convention: if $v$ is a vertex of $M(G)$, then the corresponding edge of $G$ is denoted $\overline{v}$.  We extend that to sets: if $\sigma$ is a face of $M(G)$, then $\overline{\sigma}$ is the corresponding matching of $G$.

Note that an isolated vertex of $G$ would contribute nothing to $M(G)$; we avoid them to simplify statements.  Allowing a multiple edge in a graph would duplicate a subcomplex in the matching complex.  A loop is not considered to be in any matching. 
So from now on we will assume the following:

\begin{quote}
All graphs are finite, with no isolated vertices, loops, or multiple edges.
\end{quote}

In addition, we only consider finite simplicial complexes.

Note that a matching complex $M(G)$ does not determine $G$ uniquely.
For example $G_1 = K_{1,3}$ and $G_2 = K_3$ have the same
matching complex, $M(G_i)= 3P_1$.

\begin{definition}{\em
A {\em missing face} $\sigma$  of a simplicial complex $\Delta$ is a subset
of vertices of $\Delta$ such that $\sigma$ is not a face of $\Delta$, but all
proper subsets of $\sigma$ are faces of $\Delta$.
A simplicial complex $\Delta$ is a {\em flag complex} if and only if every
missing face of $\Delta$ is of size 2.
}\end{definition}

The following proposition is proved easily from the definition.

\begin{proposition}\label{flag}
If $M(G)$ is the matching complex of a graph $G$, then $M(G)$ is a flag complex.
\end{proposition}

\textbf{Observation}.  In analogy with graph terminology, we say that
a subcomplex $N$ is an {\em induced subcomplex} of a simplicial complex $M$, if
$N$ is the restriction of $M$ to some subset of the vertices of $M$. 
If $M$ is the matching complex of a graph $G$ and $N$ is an induced
subcomplex of $M$, then $N$ is the matching complex of a subgraph of $G$, namely the subgraph spanned by the edges corresponding to vertices of $N$.

\begin{definition}{\em For a face $\sigma$ of a simplicial complex $\Delta$, the \emph{link} of $\sigma$
in $\Delta$ is
$\link_\Delta \sigma = \{ \tau \in \Delta : \tau \cup \sigma \in \Delta ~\textrm{and}~ \tau \cap \sigma = \varnothing \}.$}
\end{definition}

\begin{lemma}[Link Lemma]\label{LinkLemma}
Let $\sigma \in M(G)$. Then $\link_{M(G)} \sigma$ is the matching complex
of the subgraph of $G$ spanned by
all edges of $G$ that are not incident to any edge in $\overline{\sigma}$.
\end{lemma}

We will be referring to the subgraph mentioned in Lemma~\ref{LinkLemma} a great deal in the proofs, and we will
denote it $G_{\overline{\sigma}}$.

\vspace*{9pt}

\begin{proof}
If $\tau \in \link_{M(G)} \sigma$ then each vertex of $\tau$ forms an edge with
each vertex of $\sigma$.   So the corresponding edges form a matching of $G$,
so $\overline{\tau} \subseteq E(G_{\overline{\sigma}})$.
Thus $\link_{M(G)} \sigma \subseteq M(G_{\overline{\sigma}})$.
Similarly, given a face $\tau \in M(G_{\overline{\sigma}})$, we see that
$\overline{\tau} \cup \overline{\sigma}$ is a matching, and thus $\tau \in \link_{M(G)} \sigma$.
\end{proof}

We will often blur the distinction between the subgraph $G_{\overline{\sigma}}$
and the set of its edges.

\begin{definition}{\em
The {\em join} of two disjoint simplicial complexes $\Delta$ and $\Sigma$ is the simplicial complex $\Delta \ast \Sigma = \{\tau \cup \sigma : \mbox{$\tau\in\Delta$ and $\sigma\in\Sigma$}\}$
}\end{definition}

Joins arise in matching complexes of disconnected graphs. The following lemma follows directly from standard arguments.

\begin{lemma}[Join Lemma] \label{DisjointJoinLemma} 
Let $M(G)$ be the matching complex of a graph $G$.  Then $M(G)=M_1\ast M_2$ for some disjoint simplicial complexes $M_1$ and $M_2$ (neither equal to $\{\emptyset\}$) if and only if there exist nonempty graphs $G_1$ and $G_2$ such that $G$ is the disjoint union of $G_1$ and $G_2$,
$M_1=M(G_1)$, and $M_2=M(G_2)$.
\end{lemma}

Thus the matching complex of any disconnected graph is connected.   For what graphs are the matching complexes disconnected?  
Here we consider a matching complex to be connected if and only if the 
subcomplex consisting of all its vertices and edges (its ``1-skeleton'') 
is connected. By a ``path'' in a simplicial complex, we mean a path in its 1-skeleton.
A path in a matching complex $M(G)$ corresponds to a sequence of size two matchings in $G$ such that consecutive matchings share an edge.

The following characterization of disconnected matching complexes is equivalent
to \cite[Proposition 3.1]{JJMV}.

\begin{theorem}\label{disc-cor}
A graph $G$ has a disconnected matching complex if and only if $G=C_4$,
$G=K_4$, or $G$ has at least two edges with one edge incident to all other
edges.
\end{theorem}

\begin{proof}
($\Leftarrow$) The matching complex of $C_4$ is $2P_2$.  The matching
complex of $K_4$ is $3P_2$.  If $G$ has one edge incident to all other
edges, then the matching complex of $G$ has an isolated vertex.

($\Rightarrow$)  Assume $M(G)$ is disconnected, 
and $M_1$ is one component of $M(G)$.
Let $\{u_1,u_2,\ldots, u_k\}$ 
be the set of vertices of $M_1$, and let $\{v_1,v_2,\ldots, v_\ell\}$ 
be the remaining vertices of $M(G)$.
Consider the corresponding edges $\{\overline{u}_1,\overline{u}_2,\ldots,\overline{u}_k\}
\cup \{\overline{v}_1,\overline{v}_2,\ldots,\overline{v}_\ell\}$ of $G$.
Let $G_1$ be the subgraph of $G$ induced by the edges
$\{\overline{u}_1,\overline{u}_2,\ldots,\overline{u}_k\}$.
In $M(G)$, there are no edges between $M_1$ and the rest of $M(G)$, so
every pair of edges $\overline{u}_i$ and $\overline{v}_j$ are incident in $G$.

Suppose that no edge of $G$ is incident to all other edges. It follows that $k \ge 2$ and $\ell \ge 2$. 
Without loss of generality, assume that edges $\overline{u}_1$ and $\overline{u}_2$ are not incident. We can also assume that edges $\overline{v}_1$ and $\overline{v}_2$ are not incident, since otherwise an edge $\overline{v}_i$ would be incident to
all other edges of $G$.
Since $\overline{u}_i$ and $\overline{v}_j$ are incident for all $i$ and $j$, edges $\overline{u}_1$, $\overline{v}_1$, $\overline{u}_2$ and $\overline{v}_2$ form a cycle $C_4$. Any other edge $\overline{u}_i$ has to be incident to both $\overline{v}_1$ and $\overline{v}_2$, and any $\overline{v}_j$ has to be incident to both $\overline{u}_1$ and $\overline{u}_2$. It follows that $|V(G)| =4$.
In that case since $G$ has no edge incident to all other edges, $G$ is
$C_4$ or $K_4$.
\end{proof}

The disconnected matching complexes are all of dimension 0 or 1, and
as follows.  Note that $G\sqcup H$ denotes the disjoint union of
graphs $G$ and $H$, and $nG$ denotes $n$ disjoint copies of
the graph $G$.

\begin{itemize}\label{match-disc}
\item $nP_1$ ($n$ isolated vertices) is the matching complex of $K_{1,n}$.
\item $2P_2$ is the matching complex of $C_4$.
\item $3P_2$ is the matching complex of $K_4$.
\item $P_1 \sqcup K_{m,n}$ is the matching complex of the graph obtained from
      an edge by adding $m$ pendant edges to one of its vertices and $n$ pendant
      edges to its other vertex.
\item $P_1\sqcup (K_{m,n}\setminus H)$, where $H$ is a nonempty set of $r$
      independent edges of $K_{m,n}$, 
      is the matching complex of the graph consisting of $r$ triangles
      sharing an edge $\{x,y\}$, along with $m-r$ pendant
      edges on vertex $x$ and $n-r$ pendant edges on vertex $y$.

\end{itemize}

Jeli\'{c} Milutinovi\'{c}, et al.\ \cite{JJMV} showed that
every connected matching complex has diameter at most 4.  
We use a slightly stronger result.

\begin{proposition}\label{NoPathLemma}
If $M(G)$ is the matching complex of some graph, then $M(G)$ has no induced path
of length 5. 
\end{proposition}

\begin{proof}
Suppose $M(G)$ contains an induced $P_6$ subgraph.
Label the vertices of this  subgraph in order 1 through 6.
Then $G$ includes six edges $\overline{1}$ through
$\overline{6}$, with incidences the ten pairs of edges $\overline{i}$ and
$\overline{j}$, where $|j-i|\ge 2$.
Note that three edges in $G$ are pairwise incident if and only if the three
form a triangle or the three share a single vertex (forming a star, $K_{1,3}$).
Consider the incidences among edges
$\overline{1}$, $\overline{3}$, and $\overline{5}$. Suppose they form a triangle.  Note that
edge $\overline{2}$ is incident to edge $\overline{5}$, but not to edges $\overline{1}$
or $\overline{3}$.  This is not possible, since each vertex of $\overline{5}$ is
shared with either $\overline{1}$ or $\overline{3}$.  Thus the three edges
$\overline{1}$, $\overline{3}$, and $\overline{5}$ all meet at a single vertex $a$.
Now $\overline{4}$ is incident to $\overline{1}$, but not to $\overline{3}$ or $\overline{5}$,
and $\overline{2}$ is incident to $\overline{4}$ and $\overline{5}$, but not to $\overline{1}$
and $\overline{3}$, so the induced subgraph of $G$ on the vertices contained
in the edges $\overline{1}$ through $\overline{5}$ is a 4-cycle with a pendant edge.
This graph is called the {\em banner graph}.
(See Figure~\ref{MCP5}.)

\begin{figure}[h]
\begin{center}
\begin{tikzpicture}
\draw (2,0)--(4,0) node[draw=none,fill=none,font=\scriptsize,midway,above] {$\overline{3}$};
\draw (0,2)--(2,2) node[draw=none,fill=none,font=\scriptsize,midway,above] {$\overline{4}$};
\draw (0,0)--(2,0) node[draw=none,fill=none,font=\scriptsize,midway,above] {$\overline{5}$};
\draw (0,0)--(0,2) node[draw=none,fill=none,font=\scriptsize,midway,left] {$\overline{2}$};
\draw (2,0)--(2,2) node[draw=none,fill=none,font=\scriptsize,midway,left] {$\overline{1}$};
\filldraw[black] (2,0) circle (2pt);
\filldraw[black] (2,2) circle (2pt);
\filldraw[black] (0,2) circle (2pt);
\filldraw[black] (0,0) circle (2pt);
\filldraw[black] (4,0) circle (2pt);
\end{tikzpicture}
\caption{Banner graph $\Gamma$ whose matching complex is $P_5$\label{MCP5}}
\end{center}
\end{figure}
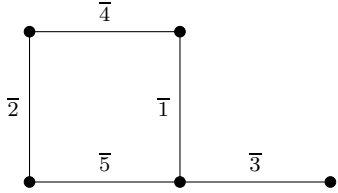

Now edge $\overline{6}$ would need to be incident to all of these edges except
$\overline{5}$, which is not possible.
So there is no graph $G$ whose matching complex has an induced $P_6$.
\end{proof}

Note that every path of length at most 4 is a matching complex, as shown
in Table~\ref{PathMatchingComplexes}.  These are all the matching
complexes that are 1-dimensional manifolds with boundary.

\begin{table}
\begin{center}
\begin{tabular}{c|c}
{\em Graph $G$} & {\em Matching Complex $M(G)$}
\\ \hline
$P_2$ & $P_1$ \\
$2P_2$ & $P_2$ \\
$P_3 \sqcup P_2$ & $P_3$ \\
$P_5$ & $P_4$ \\
$\Gamma$ & $P_5$
\end{tabular}
\caption{Matching complexes that are paths\label{PathMatchingComplexes}}
\end{center}
\end{table}

\subsection{Manifolds}

The main results of this paper concern matching complexes that are
homology spheres and manifolds.  We wish to be clear about what
these are.  We base our treatment of these definitions on \cite{KleeNovik}.

Throughout, we fix a field $\Bbbk$ and perform all homology calculations over $\Bbbk$. The following definitions depend on $\Bbbk$, but we will suppress this throughout the paper.

\begin{definition}{\em
\hspace{1in}

A \emph{homology sphere} is a $d$-dimensional simplicial complex $\Delta$ such that $\link_{\Delta}\sigma$ has the homology of a $(d-\abs{\sigma})$-sphere for all faces $\sigma \in \Delta$.

A \emph{homology manifold} (without boundary) is a $d$-dimensional simplicial complex $\Delta$ such that $\link_{\Delta}\sigma$ has the homology of a $(d-\abs{\sigma})$-sphere for all nonempty faces $\sigma \in \Delta$.

A \emph{homology manifold with boundary} is a $d$-dimensional simplicial complex $\Delta$ such that $\link_{\Delta}\sigma$ has the homology of a $(d-\abs{\sigma})$-sphere or a $(d-\abs{\sigma})$-ball for all nonempty faces $\sigma \in \Delta$, and the set $$\partial \Delta = \{ \sigma \in \Delta : \link_{\Delta}\sigma ~\text{has the homology of a}~ (d-\abs{\sigma})\text{-ball} \} \cup \{ \varnothing \}$$ is a $(d-1)$-dimensional homology manifold.

A \emph{homology ball} is a $d$-dimensional homology manifold with boundary $\Delta$ such that $\link_\Delta(\varnothing)$ has the homology of a $(d-\abs{\sigma})$-ball.
}\end{definition}

We say that $\sigma$ is in the \emph{interior} of $\Delta$ if $\link_\Delta \sigma$ is a homology sphere and on the \emph{boundary} if $\link_\Delta \sigma$ is a homology ball. For $d\ge 1$, the boundary of a homology $d$-manifold is the subcomplex generated by all $(d-1)$-faces that are contained in exactly one $d$-face.

We note that there are two related classes of simplicial complexes that are more restrictive than homology manifolds. A simplicial complex $\Delta$ is a \emph{simplicial manifold} (or \emph{triangulated manifold}) if $|\Delta| \cong M$ for some topological manifold $M$, where $|\Delta|$ is the geometric realization of $\Delta$. A complex $\Delta$ is a \emph{combinatorial manifold} if the link of each nonempty face is PL-homeomorphic to a simplex or boundary of a simplex of the appropriate dimension (see \cite{JKV} for precise definitions).

In dimensions two and lower, these three classes are equivalent. In higher dimensions, there are examples that distinguish them. It is always true that
$$
\textrm{combinatorial manifold} \implies \textrm{simplicial manifold} \implies \textrm{homology manifold}.
$$
Somewhat remarkably, we will prove that all homology manifolds that arise as matching complexes are in fact combinatorial manifolds.

A single vertex is a $0$-ball, and the two-vertex
complex, $2P_1$, is a $0$-sphere.
We will not consider $0$-dimensional complexes with a larger number of vertices
in the context of manifolds, but we have already observed that
$nP_1$ is the matching complex of $K_{1,n}$.

In what follows we will need to recognize homology manifolds with
and without boundary that are joins of lower dimensional homology manifolds.  We use
Theorem~1 from \cite{kwra}, applied in our context.

\begin{prop}\label{JoinManifold} \mbox { }
\begin{enumerate}
\item
Let $X$ be a homology $d$-manifold without boundary such that
$X = \Delta \ast \Sigma$, where $\Delta$ and $\Sigma$ are nonempty simplicial
complexes.  Then $X$ is a homology $d$-sphere, and
$\Delta$ and $\Sigma$ are homology spheres of lower dimension.
\item
Let $X$ be a homology $d$-manifold with boundary such
that ${X = \Delta \ast \Sigma}$, where $\Delta$ and $\Sigma$ are nonempty
simplicial complexes.
Then $X$ is a homology $d$-ball, and $\Delta$ and $\Sigma$ are homology spheres or homology balls,
such that at least one of $\Delta$ or $\Sigma$ is a homology ball.
\end{enumerate}
\end{prop}

Hereafter, when we refer to spheres and balls, we always assume that they are homology spheres and balls unless otherwise specified.

With Proposition~\ref{JoinManifold}
in mind, we define two important sets. The first is the set of
{\em basic sphere graphs}:
\begin{align}\label{BasicSphereGraphs}
\mathcal{SG} =  \{P_3,C_5,K_{3,2}\}.
\end{align}

Basic sphere graphs are so named because their matching complexes are spheres.
(These matching complexes are the 0-sphere, $C_5$, and $C_6$, respectively.)
The disjoint unions of these graphs give matching complexes that are higher dimensional spheres.

\begin{proposition}\label{SphereMatchingProp}
Let $G$ be a disjoint union of graphs from $\mathcal{SG}$. Then $M(G)$ is a 
combinatorial sphere.

In particular, let $G=\ell P_3 \sqcup mC_5 \sqcup nK_{3,2}$.  
Then the matching complex $M(G)$ is a combinatorial sphere of dimension
$\ell+2m+2n-1$.
\end{proposition}

\begin{proof}
The matching complex of this graph is the join of $\ell$ copies of
$S^0$ and $m+n$ copies of $S^1$.  This join is a combinatorial sphere of 
dimension $\ell+2m+2n-1$.
\end{proof}

Note, in particular, that the matching complex of $\ell P_3$ is the
boundary complex of an $\ell$-dimensional crosspolytope (generalized
octahedron).

Before defining the second set, we will note a particular sequence of graphs, generalizing $P_5$,
and their matching complexes.

\begin{definition}{\em The {\em spider} $\mbox{Sp}_k$ is the graph obtained by
identifying one end vertex of each of $k$ copies of $P_3$.
}\end{definition}

(In the literature, ``spider'' usually refers to  a more general class of
graphs, allowing legs of different lengths.)

We can visualize $\mbox{Sp}_k$ as having one central vertex $x$ with $k$
``legs'' of length~2 emanating from that vertex.
(The graph $\mbox{Sp}_2$ is just $P_5$.)
Label the edges of the $i$th leg $\overline{u}_i$ and $\overline{v}_i$, with $\overline{u}_i$
containing the central vertex $x$.
The set $\{\overline{v}_1,\overline{v}_2,\ldots, \overline{v}_{k}\}$ is a maximal
matching of $\mbox{Sp}_k$.
Since no two $\overline{u}_i$s are in any matching,  the other maximal matchings are
obtained from
$\{\overline{v}_1,\overline{v}_2,\ldots, \overline{v}_{k}\}$ by replacing a single
$\overline{v}_i$ by the neighboring $\overline{u}_i$.
The matching complex
$M(\mbox{Sp}_{k})$ thus has a central $(k-1)$-simplex
$C=\{v_1,v_2,\ldots, v_{k}\}$, and $k$ other facets containing the vertex $u_i$
and intersecting $C$ at $C\setminus\{v_i\}$.
This is a $(k-1)$-ball, thus a manifold with boundary.

We now can define the set of \emph{basic ball graphs}
\begin{align}\label{BasicBallGraphs}
\mathcal{BG} = \{ P_2, \Gamma, \mbox{Sp}_k\},
\end{align}
where $\Gamma$ is the banner graph pictured in Figure~\ref{MCP5} and $\mbox{Sp}_k$ is as defined above, for all $k\ge 2$.
We note again that $P_5 = \mbox{Sp}_2$, so this graph is contained in $\mathcal{BG}$ as well.

We have already seen that disjoint unions of
graphs from $\mathcal{SG}$ produce matching complexes that are spheres.
Here is the analogous result for balls.

\begin{prop}\label{BallProp}
Let $G$ be a disjoint union of graphs from $\mathcal{BG}$ and $\mathcal{SG}$ with at least one component from $\mathcal{BG}$. 
Then $M(G)$ is a combinatorial ball.

In particular, let $G= i P_2 \sqcup j \Gamma \sqcup \bigsqcup_{d \geq 2} k_d \mbox{\emph{Sp}}_{d} \sqcup \ell P_3 \sqcup mC_5 \sqcup nK_{3,2}$, with
$i+j+\sum k_d \ge 1$.
Then the matching complex $M(G)$ is a combinatorial ball of dimension
$i + 2j + \sum_{d} d k_d + \ell + 2m + 2n - 1$.
\end{prop}

\begin{proof}
The matching complex of this graph is the join of $i$ copies of $B^0$, $j+k_2$ copies of $B^1$, $k_d$ copies of $B^{d-1}$ (for $d >2$), $\ell$ copies of
$S^0$, and $m+n$ copies of $S^1$.  This join is a combinatorial ball of dimension $i + 2j + \sum_{d} d k_d + \ell + 2m + 2n - 1$.
\end{proof}

In the following sections, we will show that Propositions~\ref{SphereMatchingProp} and \ref{BallProp} provide the only way to construct spheres and balls as matching complexes. Moreover, outside of dimension 2, we will show that these propositions produce all possible manifold matching complexes.

\section{Low-dimensional spheres}\label{Sphere Section}

We now focus on the following question:
\begin{quote}
For which graphs $G$ is the matching complex $M(G)$ a homology sphere?
\end{quote}
We give the complete answer when $0\le d\le 2$ in this section and finish
the story for higher dimensions in Section~\ref{Manifolds_w/o_boundary Section}.
From now on the proofs generally involve detailed case analysis.  In some 
cases the reader should consult the figures to identify the labeled vertices
and edges.

Recall that  we assume our graphs are simple (having no loops or multiple
edges) and have no isolated vertices.

The 0-dimensional sphere is simply the complex
consisting of two isolated points.  As a matching complex, this
represents two edges of the graph that together do not form a matching.
In other words, $G=P_3$.

A triangulated 1-dimensional sphere is $C_n$ for some
$n\ge 3$.  We obtain the complete classification in this case.

\begin{theorem}\label{1-sphere}
The matching complexes that are 1-spheres are $C_4$, $C_5$ and $C_6$.
The only graphs giving these matching complexes are $2P_3$, $C_5$ and $K_{3,2}$,
respectively.
\end{theorem}

\begin{proof}
We consider 1-dimensional matching complexes that are $n$-cycles $C_n$
for $n\ge 3$.

${\bf n=3}$.
By Proposition \ref{flag}, a matching complex cannot be the 1-dimensional
complex $C_3$.

${\bf n=4}$. Label the vertices of $C_4$ in cyclic order 1 through 4.
If $M(G)=C_4$, then $G$ has four edges $\overline{1}$ through $\overline{4}$, with
edges $\overline{1}$ and $\overline{3}$ incident and edges $\overline{2}$ and $\overline{4}$
incident (and no other incidences among edges).  Thus $G$ is the disjoint
union of two paths, one with edges
$\overline{1}$ and $\overline{3}$, and the other with edges $\overline{2}$ and $\overline{4}$.
Thus, $M(G)=C_4$ if and only if $G=2P_3$.

${\bf n=5}$.  Label the vertices of $C_5$ in cyclic order 1 through 5.
If $M(G)=C_5$, then $G$ has five edges $\overline{1}$ through $\overline{5}$, with
incidences of exactly five pairs of edges:
edges $\overline{1}$ and $\overline{3}$;
edges $\overline{3}$ and $\overline{5}$;
edges $\overline{5}$ and $\overline{2}$;
edges $\overline{2}$ and $\overline{4}$; and
edges $\overline{4}$ and $\overline{1}$.
Thus, $M(G)=C_5$ if and only if $G=C_5$.

${\bf n=6}$.  Label the vertices of $C_6$ in cyclic order 1 through 6.
If $M(G)=C_6$, then $G$ has six edges $\overline{1}$ through
$\overline{6}$, with incidences of exactly nine pairs of edges:
edges $\overline{1}$ and $\overline{3}$;
edges $\overline{3}$ and $\overline{5}$;
edges $\overline{5}$ and $\overline{1}$;
edges $\overline{2}$ and $\overline{4}$;
edges $\overline{4}$ and $\overline{6}$;
edges $\overline{6}$ and $\overline{2}$;
edges $\overline{1}$ and $\overline{4}$;
edges $\overline{2}$ and $\overline{5}$; and
edges $\overline{3}$ and $\overline{6}$.
Recall that three edges in $G$ are pairwise incident if and only if the three
form a triangle or the three share a single vertex (forming a star, $K_{1,3}$).
As in the proof of Proposition~\ref{NoPathLemma}, the edges
$\overline{1}$, $\overline{3}$ and $\overline{5}$ cannot form a triangle, so must meet
at a single vertex $a$.  Similarly for the edges
$\overline{2}$, $\overline{4}$ and $\overline{6}$, which must meet at a single vertex $b$.
The three remaining incidences identify the endpoints (other than $a$
and $b$) of edges $\overline{1}$ and $\overline{4}$,
of edges $\overline{2}$ and $\overline{5}$, and
of edges $\overline{3}$ and $\overline{6}$.
The resulting graph is $K_{3,2}$ (Figure~\ref{MCC6}).
Thus, $M(G)=C_6$ if and only if $G=K_{3,2}$.

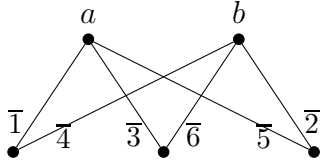
\begin{figure}[h]
\begin{center}
\begin{tikzpicture}
\node[vertex] (a) at (0,0) {};
\node[vertex] (b) at (2,0) {};
\node[vertex] (c) at (4,0) {};
\node[vertex] (d) at (1,1.5) [label={above:$a$}]{};
\node[vertex] (e) at (3,1.5) [label={above:$b$}]{};
\draw (a)--(d) node[draw=none,fill=none,font=\small,near start,left] {$\overline{1}$};
\draw (b)--(d) node[draw=none,fill=none,font=\small,very near start,left] {$\overline{3}$};
\draw (c)--(d) node[draw=none,fill=none,font=\small,very near start,left] {$\overline{5}$};
\draw (a)--(e) node[draw=none,fill=none,font=\small,very near start,right ] {$\overline{4}$};
\draw (b)--(e) node[draw=none,fill=none,font=\small,very near start,right] {$\overline{6}$};
\draw (c)--(e) node[draw=none,fill=none,font=\small,near start,right] {$\overline{2}$};
\end{tikzpicture}
\caption{Graph $K_{3,2}$ whose matching complex is $C_6$\label{MCC6}}
\end{center}
\end{figure}

${\bf n\ge 7}$.  Every cycle $C_n$ with $n\ge 7$ has an induced $P_6$
subgraph, and so is not the matching complex of a graph by
Proposition~\ref{NoPathLemma}.

Therefore, the cycles that are matching complexes are $C_4$, $C_5$,
and $C_6$, and the corresponding graphs are $2P_3$, $C_5$, and $K_{3,2}$.
\end{proof}

We note that the graphs that appear in Theorem~\ref{1-sphere} are disjoint unions of graphs from $\mathcal{SG}$, i.e., the basic sphere graphs.
In particular, we see that all $1$-spheres are constructed using Proposition~\ref{SphereMatchingProp}.

Similarly, we can use Proposition~\ref{SphereMatchingProp} to produce $2$-spheres.
That Proposition gives exactly three 2-spheres, the matching complexes of
$3P_3$, $P_3\sqcup C_5$ and $P_3\sqcup K_{3,2}$.  These graphs are the
disjoint union of $P_3$ with the graphs of Theorem~\ref{1-sphere}, so the
matching complexes are the bipyramids over $C_4$, $C_5$ and $C_6$.
The following theorem shows that this is in fact the only way to realize the $2$-sphere as a matching complex.

\begin{theorem}\label{2-sphere}
Let $G$ be a simple graph such that $M(G)$ is a 2-sphere. 
Then $G\in\{3P_3, P_3\sqcup C_5, P_3\sqcup K_{3,2}\}$ and $M(G)$ is the
boundary of the bipyramid over $C_n$ for $n\in \{4,5,6\}$.

In particular, if the matching complex of a simple graph $G$ is a
2-sphere, then $G$ is not connected.
\end{theorem}

\begin{proof}
Assume $G$ is a simple graph with no isolated vertices, and its matching
complex $M=M(G)$ is a 2-sphere.
We say that a vertex of $M$ has degree $k$ in $M$ if it is contained in
exactly $k$ edges of $M$; in this case the link
of the vertex in $M$ is a $k$-cycle.
By Lemma~\ref{LinkLemma}, the link is itself a matching complex,
and so by Theorem ~\ref{1-sphere}, $k$ must be 4, 5, or 6.
By Eberhard's Theorem (1891; see
\cite[Theorem 1 in Section 13.3]{grunbaum}) a  2-sphere
must have some vertex of degree at most 5.
We consider cases, based on the degrees of vertices.

\textbf{Case I}.  $M$ has a vertex, all of whose neighbors have degree 4.

If $M$ is not covered by Case I, and has a vertex of degree 5, then
such a vertex is adjacent to at least one vertex of degree 5 or 6.

\textbf{Case II}. $M$ has a pair of adjacent vertices, each of degree 5.

\textbf{Case III}.  $M$ has a pair of adjacent vertices, one of degree 5,
one of degree~6.

\textbf{Case IV}.  $M$ has no vertex of degree 5.

\textbf{Case I}.
Let $v$ be a vertex of $M$ such that all neighbors of
$v$ have degree 4. We show that  $M$ is the boundary of a bipyramid over
$C_4$, $C_5$, or $C_6$.

Suppose the neighbors of $v$ are (in cyclic order) $i$, $1\le i\le k$,
where $4\le k \le 6$.  
Each edge $\{i,i+1\}$ (and $\{k,1\}$) is
in a unique triangle not containing $v$.
Say that $\{k,1\}$ is in triangle $\{k,1,w\}$ ($w\ne v$).  The four
edges containing vertex $1$ are $\{v,1\}$, $\{k,1\}$, $\{w,1\}$, and $\{2,1\}$.  Each
pair of consecutive edges spans a triangle of the complex.  So
$\{1,2,w\}$ is a triangle in $M$.  Repeating this argument, we see that
for all $i$, the edge $\{i,i+1\}$ forms a triangle with this vertex $w$.
So $M$ contains the boundary of the bipyramid with base $C_n$ (vertices $i$)
and apices $v$ and $w$.  A 2-sphere cannot properly contain a 2-sphere as
a subcomplex, so $M$ is in fact the boundary of the bipyramid.

\textbf{Case II}.  Let $u$ and $v$ be adjacent vertices of $M$ of degree 5.
The links of $u$ and $v$ are induced $C_5$'s; the corresponding
subgraphs in $G$ are copies of $C_5$, which share edges $\overline{1}$ and $\overline{4}$.
See Figure~\ref{G-case2}.

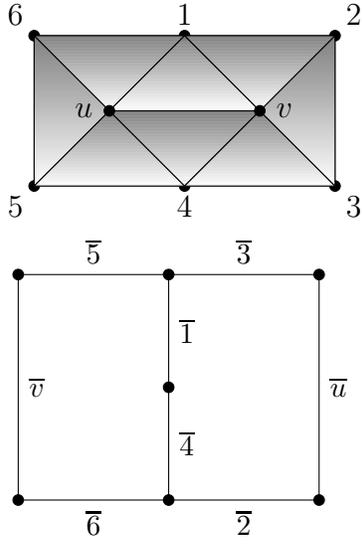
\begin{figure}[p]
\begin{center}
\begin{tikzpicture}
\draw (0,0) node[below left] {$5$}--
      (2,0) node[below] {$4$}--
      (3,1) node[right] {$v$}--
      (2,2) node[above] {$1$}--
      (0,2) node[above left] {$6$}--cycle;
\draw (2,0) --
      (4,0) node[below right] {$3$}--
      (4,2) node[above right] {$2$}--
      (2,2)--
      (1,1) node[left] {$u$}--cycle;
\draw (0,0)--(1,1)--(0,2);
\draw (4,0)--(3,1)--(4,2);
\draw (1,1)--(3,1);
\filldraw[black] (0,0) circle (2pt);
\filldraw[black] (0,2) circle (2pt);
\filldraw[black] (1,1) circle (2pt);
\filldraw[black] (2,0) circle (2pt);
\filldraw[black] (2,2) circle (2pt);
\filldraw[black] (3,1) circle (2pt);
\filldraw[black] (4,0) circle (2pt);
\filldraw[black] (4,2) circle (2pt);
\shadedraw (0,0)--(1,1)--(0,2)--(0,0);
\shadedraw (0,0)--(2,0)--(1,1)--(0,0);
\shadedraw (2,0)--(3,1)--(1,1)--(2,0);
\shadedraw (2,0)--(4,0)--(3,1)--(2,0);
\shadedraw (4,0)--(4,2)--(3,1)--(4,0);
\shadedraw (0,2)--(1,1)--(2,2)--(0,2);
\shadedraw (1,1)--(3,1)--(2,2)--(1,1);
\shadedraw (3,1)--(4,2)--(2,2)--(3,1);
\node[vertex] (u) at (1,1) [label={left:$u$}]{};
\node[vertex] (v) at (3,1) [label={right:$v$}]{};
\end{tikzpicture}

\begin{tikzpicture}
\draw (0,0)--(2,0) node[draw=none,fill=none,font=\small,midway,below] {$\overline{6}$};
\draw (2,0)--(4,0) node[draw=none,fill=none,font=\small,midway,below] {$\overline{2}$};
\draw (0,3)--(2,3) node[draw=none,fill=none,font=\small,midway,above] {$\overline{5}$};
\draw (2,3)--(4,3) node[draw=none,fill=none,font=\small,midway,above] {$\overline{3}$};
\draw (0,0)--(0,3) node[draw=none,fill=none,font=\small,midway,right] {$\overline{v}$};
\draw (2,0)--(2,1.5) node[draw=none,fill=none,font=\small,midway,right] {$\overline{4}$};
\draw (2,1.5)--(2,3) node[draw=none,fill=none,font=\small,midway,right] {$\overline{1}$};
\draw (4,0)--(4,3) node[draw=none,fill=none,font=\small,midway,right] {$\overline{u}$};
\filldraw[black] (0,0) circle (2pt);
\filldraw[black] (2,0) circle (2pt);
\filldraw[black] (4,0) circle (2pt);
\filldraw[black] (0,3) circle (2pt);
\filldraw[black] (2,3) circle (2pt);
\filldraw[black] (4,3) circle (2pt);
\filldraw[black] (2,1.5) circle (2pt);
\end{tikzpicture}
\caption{Subcomplex of $M$ and Subgraph of $G$ for Case II \label{G-case2}}
\end{center}
\end{figure}

We will show that $M$ must contain the edges $36$ and $25$. If $M$ did not contain $36$, then edges
$\overline{3}$ and $\overline{6}$ would be incident in $G$.  If the edge $\overline{6}$  were
incident with $\overline{3}$, then it would also be incident with either $\overline{1}$
or $\overline{u}$, but this is not possible, because $16$ and $6u$ are edges
of $M$.  Thus $\overline{3}$ and $\overline{6}$ form a matching in $G$, and so 36 is
an edge of $M$.  Similarly 25 is an edge of $M$.  But then the vertices
and edges  of $M$ contain a $K_5$ minor and so form a nonplanar graph.
Thus Case II cannot happen.

\textbf{Case III}.  This case is similar to Case II.
A subcomplex of $M$
and its corresponding subgraph are shown in Figure~\ref{G-case3}.

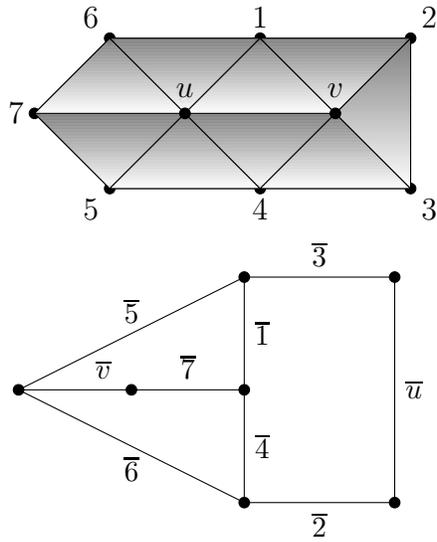
\begin{figure}[p]
\begin{center}
\begin{tikzpicture}
\draw (0,0)--(2,0) node[below] {$4$}--
      (3,1) node[right] {$v$}--
      (2,2) node[above] {$1$}--
      (0,2) node[above left] {$6$};
\draw (2,0) --
      (4,0) node[below right] {$3$}--
      (4,2) node[above right] {$2$}--
      (2,2)--
      (1,1) node[above] {$u$}--cycle;
\draw (0,0)--(1,1)--(0,2);
\draw (4,0)--(3,1)--(4,2);
\draw (1,1)--(3,1);
\draw (0,0)  node[below left] {$5$}  --(-1,1) node[left] {$7$}-- (0,2);
\draw (-1,1)--(1,1);
\filldraw[black] (0,0) circle (2pt);
\filldraw[black] (0,2) circle (2pt);
\filldraw[black] (1,1) circle (2pt);
\filldraw[black] (2,0) circle (2pt);
\filldraw[black] (2,2) circle (2pt);
\filldraw[black] (3,1) circle (2pt);
\filldraw[black] (4,0) circle (2pt);
\filldraw[black] (4,2) circle (2pt);
\filldraw[black] (-1,1) circle (2pt);
\shadedraw (0,0)--(1,1)--(-1,1)--(0,0);
\shadedraw (0,0)--(2,0)--(1,1)--(0,0);
\shadedraw (2,0)--(3,1)--(1,1)--(2,0);
\shadedraw (2,0)--(4,0)--(3,1)--(2,0);
\shadedraw (4,0)--(4,2)--(3,1)--(4,0);
\shadedraw (-1,1)--(1,1)--(0,2)--(-1,1);
\shadedraw (1,1)--(2,2)--(0,2)--(1,1);
\shadedraw (1,1)--(3,1)--(2,2)--(1,1);
\shadedraw (3,1)--(4,2)--(2,2)--(3,1);
\node[vertex] (u) at (1,1) [label={above:$u$}]{};
\node[vertex] (v) at (3,1) [label={above:$v$}]{};
\end{tikzpicture}

\begin{tikzpicture}
\draw (2,0)--(4,0) node[draw=none,fill=none,font=\small,midway,below] {$\overline{2}$};
\draw (2,3)--(4,3) node[draw=none,fill=none,font=\small,midway,above] {$\overline{3}$};
\draw (2,0)--(2,1.5) node[draw=none,fill=none,font=\small,midway,right] {$\overline{4}$};
\draw (2,1.5)--(2,3) node[draw=none,fill=none,font=\small,midway,right] {$\overline{1}$};
\draw (4,0)--(4,3) node[draw=none,fill=none,font=\small,midway,right] {$\overline{u}$};
\draw (-1,1.5)--(2,3) node[draw=none,fill=none,font=\small,midway,above] {$\overline{5}$};
\draw (-1,1.5)--(0.5,1.5) node[draw=none,fill=none,font=\small,near end,above] {$\overline{v}$};
\draw (0.5,1.5)--(2,1.5) node[draw=none,fill=none,font=\small,midway,above] {$\overline{7}$};
\draw (-1,1.5)--(2,0) node[draw=none,fill=none,font=\small,midway,below] {$\overline{6}$};
\filldraw[black] (2,0) circle (2pt);
\filldraw[black] (4,0) circle (2pt);
\filldraw[black] (2,3) circle (2pt);
\filldraw[black] (4,3) circle (2pt);
\filldraw[black] (2,1.5) circle (2pt);
\filldraw[black] (0.5,1.5) circle (2pt);
\filldraw[black] (-1,1.5) circle (2pt);
\end{tikzpicture}
\caption{Subcomplex of $M$ and Subgraph of $G$ for Case III\label{G-case3}}
\end{center}
\end{figure}

The link of $v$ is an induced $C_5$, with corresponding graph $C_5$;
the link of  $u$ is an induced $C_6$, with corresponding graph $K_{3,2}$.
These subgraphs in $G$ share edges $\overline{1}$ and $\overline{4}$.

Just as in Case II, 25 and 36 must be edges of $M$, and the vertices and edges of
$M$ then form a nonplanar graph.  So Case III cannot happen.

\textbf{Case IV}. By Eberhard's Theorem $M$, having no vertices of degree 3 or 5,
must have at least one vertex $v$ of degree 4.
By Case I, we need only consider the subcase where $v$ has at least
one neighbor $u$ of degree 6.  A subcomplex of $M$
and its corresponding subgraph are shown in Figure~\ref{MC-case4}.
(The vertices
are labeled $u$, $v$, and 1 through 7, except 3, to most closely match Case II.)

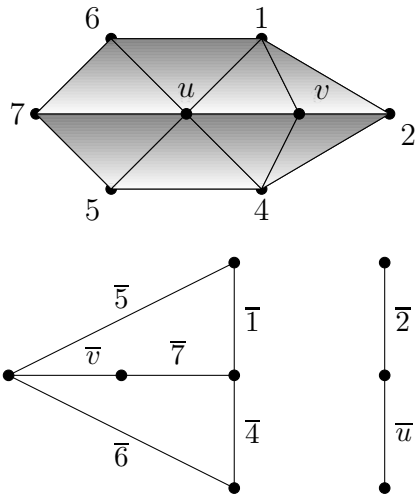
\begin{figure}[h]
\begin{center}
\begin{tikzpicture}
\draw (0,0)--(2,0) node[below] {$4$}--
      (2.5,1) node[above right] {$v$}--
      (2,2) node[above] {$1$}--
      (0,2) node[above left] {$6$};
\draw (2,0) --
      (3.7,1) node[below right] {$2$}--
      (2,2)--
      (1,1) node[above] {$u$}--cycle;
\draw (0,0)--(1,1)--(0,2);
\draw (1,1)--(2.5,1);
\draw (0,0)  node[below left] {$5$}  --(-1,1) node[left] {$7$}-- (0,2);
\draw (-1,1)--(1,1);
\draw (2.5,1)--(3.7,1);
\filldraw[black] (0,0) circle (2pt);
\filldraw[black] (0,2) circle (2pt);
\filldraw[black] (1,1) circle (2pt);
\filldraw[black] (2,0) circle (2pt);
\filldraw[black] (2,2) circle (2pt);
\filldraw[black] (2.5,1) circle (2pt);
\filldraw[black] (3.7,1) circle (2pt);
\filldraw[black] (-1,1) circle (2pt);
\shadedraw (0,0)--(1,1)--(-1,1)--(0,0);
\shadedraw (0,0)--(2,0)--(1,1)--(0,0);
\shadedraw (2,0)--(2.5,1)--(1,1)--(2,0);
\shadedraw (2,0)--(3.7,1)--(2.5,1)--(2,0);
\shadedraw (1,1)--(0,2)--(-1,1)--(1,1);
\shadedraw (1,1)--(2,2)--(0,2)--(1,1);
\shadedraw (1,1)--(2.5,1)--(2,2)--(1,1);
\shadedraw (2.5,1)--(3.7,1)--(2,2)--(2.5,1);
\node[vertex] (v) at (1,1) [label={above:$u$}]{};
\node[vertex] (v) at (2.5,1) [label={above right:$v$}]{};
\end{tikzpicture}

\vspace{9pt}

\begin{tikzpicture}
\draw (2,0)--(2,1.5) node[draw=none,fill=none,font=\small,midway,right] {$\overline{4}$};
\draw (2,1.5)--(2,3) node[draw=none,fill=none,font=\small,midway,right] {$\overline{1}$};
\draw (4,0)--(4,1.5) node[draw=none,fill=none,font=\small,midway,right] {$\overline{u}$};
\draw (4,1.5)--(4,3) node[draw=none,fill=none,font=\small,midway,right] {$\overline{2}$};
\draw (-1,1.5)--(2,3) node[draw=none,fill=none,font=\small,midway,above] {$\overline{5}$};
\draw (-1,1.5)--(0.5,1.5) node[draw=none,fill=none,font=\small,near end,above] {$\overline{v}$};
\draw (0.5,1.5)--(2,1.5) node[draw=none,fill=none,font=\small,midway,above] {$\overline{7}$};
\draw (-1,1.5)--(2,0) node[draw=none,fill=none,font=\small,midway,below] {$\overline{6}$};
\filldraw[black] (2,0) circle (2pt);
\filldraw[black] (4,0) circle (2pt);
\filldraw[black] (2,3) circle (2pt);
\filldraw[black] (4,3) circle (2pt);
\filldraw[black] (2,1.5) circle (2pt);
\filldraw[black] (0.5,1.5) circle (2pt);
\filldraw[black] (-1,1.5) circle (2pt);
\filldraw[black] (4,1.5) circle (2pt);
\end{tikzpicture}
\caption{Subcomplex of $M$ and Subgraph of $G$ for Case IV\label{MC-case4}}
\end{center}
\end{figure}

The link of $v$ is an induced $C_4$, with
corresponding graph $2P_3$; the link of $u$ is an induced $C_6$, with
corresponding graph $K_{3,2}$.  These subgraphs in $G$ share edges $\overline{1}$
and $\overline{4}$.

In $M$, vertex 2 is adjacent to vertices 1, 4 and $v$, so in $G$ edge
$\overline{2}$ is not incident to edges $\overline{1}$, $\overline{4}$ and $\overline{v}$.
But then it cannot be incident to any of the edges $\overline{5}$, $\overline{6}$ and
$\overline{7}$.  So $M$ must have edges 25, 26 and 27.  Then $M$ contains vertices
and edges forming the graph of a bipyramid over the hexagon $16754v$.   Since
$M$ is flag, it contains the whole boundary of the bipyramid as a subcomplex.
As noted before, this implies that all of $M$ is the boundary of the
bipyramid.

So in all cases, if $M$ is the matching complex of a simple graph (with no
isolated vertices) and $M$ is a 2-sphere, $M$ must be
the boundary of a bipyramid over a $k$-gon, $k\in\{4,5,6\}$.  We have
already seen that the graphs that
give these matching complexes are $3P_3$, $P_3\sqcup C_5$, and
$P_3\sqcup K_{3,2}$.
\end{proof}

By this theorem Proposition~\ref{SphereMatchingProp} gives the only way to realize the $2$-sphere as a matching complex, since the graphs in Theorem~\ref{2-sphere} are disjoint unions of elements of $\mathcal{SG}$.

Determining which spheres are matching complexes in higher dimension (aside from those of
Proposition~\ref{SphereMatchingProp}) is more complicated.
We approach this problem by considering the more general question of which matching complexes are homology manifolds.

\section{Manifolds without boundary}\label{Manifolds_w/o_boundary Section}

In this section, we answer the following questions:

\begin{quote}
For which graphs $G$ is the matching complex $M(G)$ a homology manifold?
Given a homology manifold, is it the matching complex of a graph?
\end{quote}

As before, all graphs are simple (without loops or multiple edges) and do not have any isolated vertices.

We will use the following standard observation: If $X$ and $Y$ are closed,
connected  homology $d$-manifolds without boundary, and $Y\subseteq X$, then $X=Y$. 
That is, no proper, connected, full-dimensional subcomplex of a closed manifold 
without boundary is a manifold without boundary.

Disconnected matching complexes were classified in
Corollary~\ref{disc-cor}.
None of these matching complexes (of dimension greater than 1) are manifolds without
boundary, so we can restrict ourselves to connected homology manifolds.

Since the only closed $d$-manifolds without boundary for $d<2$ are spheres (and disjoint unions of spheres), cases $d=0$ and $d=1$ are answered in Section \ref{Sphere Section}.
We summarize the results below. Throughout, $M(G)$ is a manifold without boundary and $G$ is a simple graph.

\begin{itemize}
\item Let $\dim M(G) = 0$. Then $G=P_2$, $G=C_3$ or
      $G=K_{1,n}$ for any $n\ge 2$.
\item Let $\dim M(G) = 1$. Then $G \in \{ 2P_3, C_5, K_{3,2} \}.$
\end{itemize}

However, when $\dim M=2$, the situation becomes more complex.

\begin{proposition}\emph{\cite[Page 30]{blvz}} \label{TorusMatchingProp}
If $G = K_{4,3}$, then $M(G)$ is a triangulation of  $T^2$, 
the (2-dimensional) torus.
\end{proposition}

Figure~\ref{TorusMatchingFig} shows the labeled graph $K_{4,3}$
and its matching complex. We can see that the top and bottom edges of the 
matching complex diagram are identified with the same orientation; the same is true for the left and right sides. Therefore the matching complex of $K_{4,3}$ is a triangulation of $T^2$, the torus.

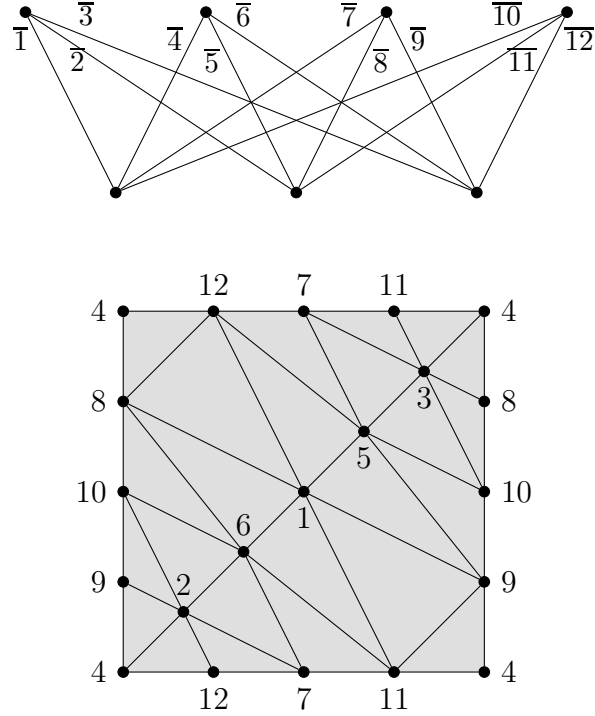
\begin{figure}[h]
\begin{center}
\begin{tikzpicture}[scale=1.2]
\node[vertex] (a) at (1,0) {};
\node[vertex] (b) at (3,0) {};
\node[vertex] (c) at (5,0) {};
\node[vertex] (d) at (0,2) {};
\node[vertex] (e) at (2,2) {};
\node[vertex] (f) at (4,2) {};
\node[vertex] (g) at (6,2) {};
\draw (d)--(a) node[draw=none,fill=none,font=\small,very near start,left] {$\overline{1}$};
\draw (d)--(b) node[draw=none,fill=none,font=\small,near start,left] {$\overline{2}$};
\draw (d)--(c) node[draw=none,fill=none,font=\small,very near start,above] {$\overline{3}$};
\draw (e)--(a) node[draw=none,fill=none,font=\small,very near start,left] {$\overline{4}$};
\draw (e)--(b) node[draw=none,fill=none,font=\small,near start,left] {$\overline{5}$};
\draw (e)--(c) node[draw=none,fill=none,font=\small,very near start,above] {$\overline{6}$};
\draw (f)--(a) node[draw=none,fill=none,font=\small,very near start,above] {$\overline{7}$};
\draw (f)--(b) node[draw=none,fill=none,font=\small,near start,right] {$\overline{8}$};
\draw (f)--(c) node[draw=none,fill=none,font=\small,very near start,right] {$\overline{9}$};
\draw (g)--(a) node[draw=none,fill=none,font=\small,very near start,above] {$\overline{10}$};
\draw (g)--(b) node[draw=none,fill=none,font=\small,near start,right] {$\overline{11}$};
\draw (g)--(c) node[draw=none,fill=none,font=\small,very near start,right] {$\overline{12}$};
\end{tikzpicture}

\vspace{24pt}

\begin{tikzpicture}[scale=1.2]

\path[fill=gray!25] (0,0) -- (4,0) -- (4,4) -- (0,4) -- cycle; 

\node[vertex] (4a) at (0,4) [label={left:$4$}]{};
\node[vertex] (8a) at (0,3) [label={left:$8$}]{};
\node[vertex] (10a) at (0,2) [label={left:$10$}]{};
\node[vertex] (9a) at (0,1) [label={left:$9$}]{};
\node[vertex] (4b) at (0,0) [label={left:$4$}]{};
\node[vertex] (4c) at (4,4) [label={right:$4$}]{};
\node[vertex] (8b) at (4,3) [label={right:$8$}]{};
\node[vertex] (10b) at (4,2) [label={right:$10$}]{};
\node[vertex] (9b) at (4,1) [label={right:$9$}]{};
\node[vertex] (4d) at (4,0) [label={right:$4$}]{};
\node[vertex] (12a) at (1,4) [label={above:$12$}]{};
\node[vertex] (7a) at (2,4) [label={above:$7$}]{};
\node[vertex] (11a) at (3,4) [label={above:$11$}]{};
\node[vertex] (12b) at (1,0) [label={below:$12$}]{};
\node[vertex] (7b) at (2,0) [label={below:$7$}]{};
\node[vertex] (11b) at (3,0) [label={below:$11$}]{};
\node[vertex] (2) at (.667,.667) [label={above:$2$}]{};
\node[vertex] (6) at (1.333,1.333) [label={above:$6$}]{};
\node[vertex] (1) at (2,2) [label={below:$1$}]{};
\node[vertex] (5) at (2.667,2.667) [label={below:$5$}]{};
\node[vertex] (3) at (3.333,3.333) [label={below:$3$}]{};
\draw (4a)--(4b) node[draw=none,fill=none,font=\small,very near start,above] {};
\draw (4a)--(4c) node[draw=none,fill=none,font=\small,very near start,above] {};
\draw (4b)--(4d) node[draw=none,fill=none,font=\small,very near start,above] {};
\draw (4c)--(4d) node[draw=none,fill=none,font=\small,very near start,above] {};
\draw (4b)--(4c) node[draw=none,fill=none,font=\small,very near start,above] {};
\draw (8a)--(12a) node[draw=none,fill=none,font=\small,very near start,above] {};
\draw (11b)--(9b) node[draw=none,fill=none,font=\small,very near start,above] {};
\draw (2)--(9a) node[draw=none,fill=none,font=\small,very near start,above] {};
\draw (2)--(10a) node[draw=none,fill=none,font=\small,very near start,above] {};
\draw (2)--(7b) node[draw=none,fill=none,font=\small,very near start,above] {};
\draw (2)--(12b) node[draw=none,fill=none,font=\small,very near start,above] {};
\draw (6)--(10a) node[draw=none,fill=none,font=\small,very near start,above] {};
\draw (6)--(8a) node[draw=none,fill=none,font=\small,very near start,above] {};
\draw (6)--(11b) node[draw=none,fill=none,font=\small,very near start,above] {};
\draw (6)--(7b) node[draw=none,fill=none,font=\small,very near start,above] {};
\draw (1)--(8a) node[draw=none,fill=none,font=\small,very near start,above] {};
\draw (1)--(12a) node[draw=none,fill=none,font=\small,very near start,above] {};
\draw (1)--(9b) node[draw=none,fill=none,font=\small,very near start,above] {};
\draw (1)--(11b) node[draw=none,fill=none,font=\small,very near start,above] {};
\draw (5)--(12a) node[draw=none,fill=none,font=\small,very near start,above] {};
\draw (5)--(7a) node[draw=none,fill=none,font=\small,very near start,above] {};
\draw (5)--(10b) node[draw=none,fill=none,font=\small,very near start,above] {};
\draw (5)--(9b) node[draw=none,fill=none,font=\small,very near start,above] {};
\draw (3)--(7a) node[draw=none,fill=none,font=\small,very near start,above] {};
\draw (3)--(11a) node[draw=none,fill=none,font=\small,very near start,above] {};
\draw (3)--(8b) node[draw=none,fill=none,font=\small,very near start,above] {};
\draw (3)--(10b) node[draw=none,fill=none,font=\small,very near start,above] {};
\end{tikzpicture}
\caption{The graph $K_{4,3}$ and its matching complex.
Edges with identical vertex labels are identified.
\label{TorusMatchingFig}}
\end{center}
\end{figure}

When $M(G)$ is two dimensional, the next theorem shows that there are two options: $M(G)$ triangulates either the torus $T^2$ or the sphere ${S}^2$.

\begin{theorem}\label{2dimMan} 
Let $G$ be a simple graph such that $M(G)$ is a 2-dimensional manifold.
Then either 
\begin{enumerate}
\item $G = K_{4,3}$ and $M(G)$ is a triangulation of $T^2$, or
\item $G \in \{ 3P_3, P_3 \sqcup C_5, P_3 \sqcup K_{3,2}\}$ and $M(G)$ is a
triangulation of ${S}^2$.
\end{enumerate}

\end{theorem}

\begin{proof}
If $G$ is disconnected, then $G = G_1 \sqcup G_2$ and thus $M(G) = M(G_1) \ast M(G_2)$ by Lemma~\ref{DisjointJoinLemma}. Therefore by Proposition \ref{JoinManifold}, $M(G)$ is a sphere.
This case is covered by Theorem~\ref{2-sphere}.

Otherwise assume that $G$ is connected. We will show that the only possibility in this case is that $G = K_{4,3}$. Since $M = M(G)$ is a $2$-manifold, if $v$ is a vertex of $M$, then $\link_M v$ is a $1$-sphere. Furthermore, for any face $\sigma \in M$, $\link_M \sigma = M(G_{\overline{\sigma}})$ by Lemma \ref{LinkLemma}. Therefore, $\link_Mv$ is either $C_4$, $C_5$, or $C_6$ by Theorem \ref{1-sphere}. We will use this to consider three different cases.

\textbf{Case I}. All vertices $v \in M$ have $\link_M v = C_4.$

\textbf{Case II}. There exists a vertex $v \in M$ such that $\link_M v = C_5.$

\textbf{Case III}. There exists a vertex $v \in M$ such that $\link_M v = C_6.$

\textbf{Case I}. In this case, $G_{\overline{v}} = 2P_3$ for all vertices $v \in M$. Let $v$ be some vertex of $M$. Then $G$ must contain a subgraph as in Figure~\ref{2Dim-CaseI}.

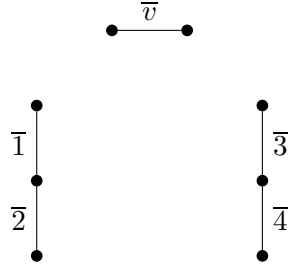
\begin{figure}[h]
\begin{center}
\begin{tikzpicture}
\node[vertex] (a) at (0,0) {};
\node[vertex] (b) at (0,1) {};
\node[vertex] (c) at (0,2) {};
\node[vertex] (d) at (1,3) {};
\node[vertex] (e) at (2,3) {};
\node[vertex] (f) at (3,2) {};
\node[vertex] (g) at (3,1) {};
\node[vertex] (h) at (3,0) {};
\draw (a)--(b) node[draw=none,fill=none,font=\small,midway,left] {$\overline{2}$};
\draw (b)--(c) node[draw=none,fill=none,font=\small,midway,left] {$\overline{1}$};
\draw (d)--(e) node[draw=none,fill=none,font=\small,midway,above] {$\overline{v}$};
\draw (h)--(g) node[draw=none,fill=none,font=\small,midway,right] {$\overline{4}$};
\draw (g)--(f) node[draw=none,fill=none,font=\small,midway,right] {$\overline{3}$};
\end{tikzpicture}
\caption{A subgraph for Case I of Theorem~4.2. All remaining edges of $G$ must share an endpoint with $\overline{v}$. \label{2Dim-CaseI}}
\end{center}
\end{figure}

$G$ is connected, and all remaining edges of $G$ must have a vertex in common with $\overline{v}$. Therefore there must be an edge $\overline{a}$ connecting $\overline{v}$ and the component containing the edges $\overline{3}$ and $\overline{4}$. Thus $G_{\overline{1}}$ contains the edges $\overline{a}$, $\overline{v}$, $\overline{3}$, and $\overline{4}$ and therefore contains a path of length three. But $G_{\overline{1}}$ is $2P_3$ by assumption, so this is a contradiction. Therefore Case I is not possible.

\textbf{Case II}. In this case, $G$ must contain an edge $\overline{v}$ and $C_5$ that is disjoint from $\overline{v}$. Since $G$ is connected and every other edge of $G$ shares a vertex with $\overline{v}$, we assume without loss of generality that there is an edge $\overline{a}$ as in Figure~\ref{2Dim-CaseII1}.

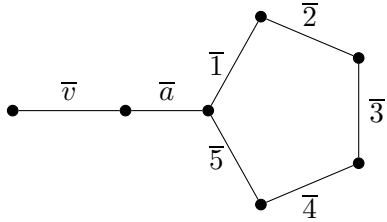
\begin{figure}[h]
\begin{center}
\begin{tikzpicture}
\node[vertex] (m) at (-.35,1.5) {};
\node[vertex] (n) at (1.15,1.5) {};
\node[vertex] (o) at (2.25,1.5) {};
\node[vertex] (p) at (4.25,2.2) {};
\node[vertex] (q) at (2.95,2.75) {};
\node[vertex] (r) at (2.95,.25) {};
\node[vertex] (s) at (4.25,.8) {};
\draw (m)--(n) node[draw=none,fill=none,font=\small,midway,above ] {$\overline{v}$};
\draw (o)--(q) node[draw=none,fill=none,font=\small,midway, left] {$\overline{1}$};
\draw (q)--(p) node[draw=none,fill=none,font=\small,midway,above ] {$\overline{2}$};
\draw (p)--(s) node[draw=none,fill=none,font=\small,midway,right] {$\overline{3}$};
\draw (s)--(r) node[draw=none,fill=none,font=\small,midway,below] {$\overline{4}$};
\draw (r)--(o) node[draw=none,fill=none,font=\small,midway,left] {$\overline{5}$};
\draw (n)--(o) node[draw=none,fill=none,font=\small,midway,above] {$\overline{a}$};
\end{tikzpicture}
\caption{A subgraph for Case II of Theorem~4.2. All remaining edges of $G$ must share an endpoint with $\overline{v}$. \label{2Dim-CaseII1}}
\end{center}
\end{figure}

Now consider $G_{\overline{3}}$. Since it already contains $\overline{v},$ $\overline{a},$ $\overline{1},$ and $\overline{5}$, the only possibility is that $G_{\overline{3}} = K_{3,2}$ by Theorem~\ref{1-sphere}. Thus there are also the edges $\overline{b}$ and $\overline{c}$ in $G$ as shown in Figure~\ref{2Dim-CaseII2}.

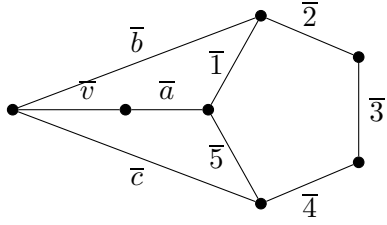
\begin{figure}[h]
\begin{center}
\begin{tikzpicture}
\node[vertex] (m) at (-.35,1.5) {};
\node[vertex] (n) at (1.15,1.5) {};
\node[vertex] (o) at (2.25,1.5) {};
\node[vertex] (p) at (4.25,2.2) {};
\node[vertex] (q) at (2.95,2.75) {};
\node[vertex] (r) at (2.95,.25) {};
\node[vertex] (s) at (4.25,.8) {};
\draw (m)--(n) node[draw=none,fill=none,font=\small,midway,above right] {$\overline{v}$};
\draw (o)--(q) node[draw=none,fill=none,font=\small,midway, left] {$\overline{1}$};
\draw (q)--(p) node[draw=none,fill=none,font=\small,midway,above ] {$\overline{2}$};
\draw (p)--(s) node[draw=none,fill=none,font=\small,midway,right] {$\overline{3}$};
\draw (s)--(r) node[draw=none,fill=none,font=\small,midway,below] {$\overline{4}$};
\draw (r)--(o) node[draw=none,fill=none,font=\small,midway,left] {$\overline{5}$};
\draw (n)--(o) node[draw=none,fill=none,font=\small,midway,above] {$\overline{a}$};
\draw (m)--(q) node[draw=none,fill=none,font=\small,midway,above] {$\overline{b}$};
\draw (m)--(r) node[draw=none,fill=none,font=\small,midway,below] {$\overline{c}$};
\end{tikzpicture}
\caption{A subgraph for Case II of Theorem~4.2, with edges added from $G_{\overline{3}}$. \label{2Dim-CaseII2}}
\end{center}
\end{figure}

Similarly, we now consider $G_{\overline{2}}$ and $G_{\overline{4}}$ separately. By the same reasoning as for $G_{\overline{3}}$ above, both of these subgraphs must be $K_{3,2}$, which gives us the new edges $\overline{d}$ and $\overline{e}$ in Figure~\ref{2Dim-CaseII3}.

\begin{figure}[h]
\begin{center}
\begin{tikzpicture}
\node[vertex] (m) at (-.35,1.5) {};
\node[vertex] (n) at (1.15,1.5) {};
\node[vertex] (o) at (2.25,1.5) {};
\node[vertex] (p) at (4.25,2.2) {};
\node[vertex] (q) at (2.95,2.75) {};
\node[vertex] (r) at (2.95,.25) {};
\node[vertex] (s) at (4.25,.8) {};
\draw (m)--(n) node[draw=none,fill=none,font=\small,midway,above right] {$\overline{v}$};
\draw (o)--(q) node[draw=none,fill=none,font=\small,midway, left] {$\overline{1}$};
\draw (q)--(p) node[draw=none,fill=none,font=\small,midway,above ] {$\overline{2}$};
\draw (p)--(s) node[draw=none,fill=none,font=\small,midway,right] {$\overline{3}$};
\draw (s)--(r) node[draw=none,fill=none,font=\small,midway,below] {$\overline{4}$};
\draw (r)--(o) node[draw=none,fill=none,font=\small,midway,left] {$\overline{5}$};
\draw (n)--(o) node[draw=none,fill=none,font=\small,midway,above] {$\overline{a}$};
\draw (m)--(q) node[draw=none,fill=none,font=\small,midway,above] {$\overline{b}$};
\draw (m)--(r) node[draw=none,fill=none,font=\small,midway,below] {$\overline{c}$};
\draw (n)--(p) node[draw=none,fill=none,font=\small,midway,above right] {$\overline{d}$};
\draw (n)--(s) node[draw=none,fill=none,font=\small,midway,above right] {$\overline{e}$};
\end{tikzpicture}
\caption{A subgraph for Case II of Theorem~4.2, with edges from $G_{\overline{2}}$ and $G_{\overline{4}}$. \label{2Dim-CaseII3}}
\end{center}
\end{figure}
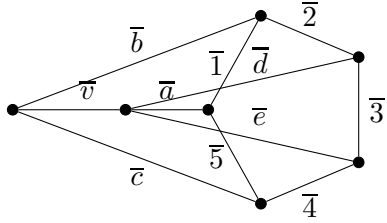
Now we can see that the subgraph $G_{\overline{1}}$ must contain the triangle
$\overline{d}\overline{e}\overline{3}$, but this is a contradiction by Theorem~\ref{1-sphere}.
Therefore Case II is also impossible.

\textbf{Case III}. In this case, $G$ must contain an edge $\overline{v}$ and $K_{3,2}$ that is disjoint from $\overline{v}$. Recall that $G$ is connected and every other edge of $G$ shares a vertex with $\overline{v}$. There must exist a subgraph of $G$ as in Figure~\ref{2Dim-CaseIII}, and all other edges of $G$ must share an endpoint with $\overline{v}$. We will split this case up into two subcases.

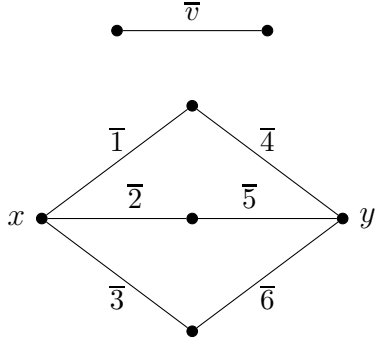
\begin{figure}[h]
\begin{center}
\begin{tikzpicture}
\node[vertex] (x) at (0,1.5) [label={left:$x$}]{};
\node[vertex] (y) at (4,1.5) [label={right:$y$}]{};
\node[vertex] (z) at (3,4) {};
\node[vertex] (m) at (1,4) {};
\node[vertex] (n) at (2,3) {};
\node[vertex] (o) at (2,1.5) {};
\node[vertex] (p) at (2,0) {};
\draw (m)--(z) node[draw=none,fill=none,font=\small,midway,above ] {$\overline{v}$};
\draw (x)--(n) node[draw=none,fill=none,font=\small,midway,above ] {$\overline{1}$};
\draw (x)--(o) node[draw=none,fill=none,font=\small,midway,above right ] {$\overline{2}$};
\draw (x)--(p) node[draw=none,fill=none,font=\small,midway,below ] {$\overline{3}$};
\draw (y)--(n) node[draw=none,fill=none,font=\small,midway,above ] {$\overline{4}$};
\draw (y)--(o) node[draw=none,fill=none,font=\small,midway,above left ] {$\overline{5}$};
\draw (y)--(p) node[draw=none,fill=none,font=\small,midway,below ] {$\overline{6}$};
\end{tikzpicture}
\caption{A subgraph for Case III of Theorem~4.2. All remaining edges of $G$ must share an endpoint with $\overline{v}$. \label{2Dim-CaseIII}}
\end{center}
\end{figure}

\textbf{Case III.1}. There are no edges between the endpoints of $\overline{v}$ and vertices $x$ and $y$.

Since $G$ is connected, there must be an edge $\overline{a}$ connecting one of the endpoints of $\overline{v}$ with one of the middle vertices in the copy of $K_{3,2}$ in Figure~\ref{2Dim-CaseIII}. Without loss of generality, we assume it is as in Figure~\ref{2Dim-CaseIII1}.

\begin{figure}[h]
\begin{center}
\begin{tikzpicture}
\node[vertex] (x) at (0,1.5) [label={left:$x$}]{};
\node[vertex] (y) at (4,1.5) [label={right:$y$}]{};
\node[vertex] (z) at (3,4) {};
\node[vertex] (m) at (1,4) {};
\node[vertex] (n) at (2,3) {};
\node[vertex] (o) at (2,1.5) {};
\node[vertex] (p) at (2,0) {};
\draw (m)--(z) node[draw=none,fill=none,font=\small,midway,above ] {$\overline{v}$};
\draw (x)--(n) node[draw=none,fill=none,font=\small,midway,above ] {$\overline{1}$};
\draw (x)--(o) node[draw=none,fill=none,font=\small,midway,above right ] {$\overline{2}$};
\draw (x)--(p) node[draw=none,fill=none,font=\small,midway,below ] {$\overline{3}$};
\draw (y)--(n) node[draw=none,fill=none,font=\small,midway,above ] {$\overline{4}$};
\draw (y)--(o) node[draw=none,fill=none,font=\small,midway,above left ] {$\overline{5}$};
\draw (y)--(p) node[draw=none,fill=none,font=\small,midway,below ] {$\overline{6}$};
\draw (m)--(n) node[draw=none,fill=none,font=\small,midway,right ] {$\overline{a}$};
\end{tikzpicture}
\caption{A subgraph for Case III.1 of Theorem~4.2. All remaining edges of $G$ must share an endpoint with $\overline{v}$. \label{2Dim-CaseIII1}}
\end{center}
\end{figure}

Therefore $G_{\overline{a}}$ contains a four cycle. By Theorem~\ref{1-sphere}, this implies that $G_{\overline{a}} = K_{3,2}$. Since all remaining edges in $G$ must share an endpoint with $\overline{v}$, the only possibility in this case is to add the edges $\overline{b}$ and $\overline{c}$ as in Figure~\ref{2Dim-CaseIII1a}.

\begin{figure}[h]
\begin{center}
\begin{tikzpicture}
\node[vertex] (x) at (0,1.5) [label={left:$x$}]{};
\node[vertex] (y) at (4,1.5) [label={right:$y$}]{};
\node[vertex] (z) at (3.5,4) {};
\node[vertex] (m) at (1,4) {};
\node[vertex] (n) at (2,3) {};
\node[vertex] (o) at (2,1.5) {};
\node[vertex] (p) at (2,0) {};
\draw (m)--(z) node[draw=none,fill=none,font=\small,midway,above ] {$\overline{v}$};
\draw (x)--(n) node[draw=none,fill=none,font=\small,midway,above ] {$\overline{1}$};
\draw (x)--(o) node[draw=none,fill=none,font=\small,midway,above right ] {$\overline{2}$};
\draw (x)--(p) node[draw=none,fill=none,font=\small,midway,below ] {$\overline{3}$};
\draw (y)--(n) node[draw=none,fill=none,font=\small,midway,above right ] {$\overline{4}$};
\draw (y)--(o) node[draw=none,fill=none,font=\small,midway,below ] {$\overline{5}$};
\draw (y)--(p) node[draw=none,fill=none,font=\small,midway,below ] {$\overline{6}$};
\draw (m)--(n) node[draw=none,fill=none,font=\small,midway,right ] {$\overline{a}$};
\draw (z)--(o) node[draw=none,fill=none,font=\small,midway,above ] {$\overline{b}$};
\draw (z)--(p) node[draw=none,fill=none,font=\small,midway,right ] {$\overline{c}$};
\end{tikzpicture}
\caption{A subgraph for Case III.1 of Theorem~4.2. \label{2Dim-CaseIII1a}}
\end{center}
\end{figure}

Now considering $G_{\overline{1}}$, we again see that this subgraph must be $K_{3,2}$ by Theorem~\ref{1-sphere}. Therefore there must be an edge connecting $y$ and the left endpoint of $\overline{v}$. But this contradicts our assumption, so this case is not possible.

\textbf{Case III.2}. Edges between the endpoints of $\overline{v}$ and $x$ and $y$ are allowed.

Assume without loss of generality that the edge $\overline{a}$ is in $G$ as depicted in Figure~\ref{2Dim-CaseIII2}. Considering $G_{\overline{6}}$, we see that we must have edges $\overline{b}$ and $\overline{c}$ depicted in Figure~\ref{2Dim-CaseIII2a}, as $G_{\overline{6}} = K_{3,2}$ by Theorem~\ref{1-sphere}. Similarly, Theorem~\ref{1-sphere} shows that $G_{\overline{5}}$ and then $G_{\overline{1}}$ must also be $K_{3,2}$, which gives us edges $\overline{d}$ and $\overline{e}$, respectively, as in Figure~\ref{2Dim-CaseIII2b}.

\begin{figure}[p]
\begin{center}
\begin{tikzpicture}
\node[vertex] (x) at (0,1.5) [label={left:$x$}]{};
\node[vertex] (y) at (4,1.5) [label={right:$y$}]{};
\node[vertex] (z) at (3,4) {};
\node[vertex] (m) at (1,4) {};
\node[vertex] (n) at (2,3) {};
\node[vertex] (o) at (2,1.5) {};
\node[vertex] (p) at (2,0) {};
\draw (m)--(z) node[draw=none,fill=none,font=\small,midway,above ] {$\overline{v}$};
\draw (x)--(n) node[draw=none,fill=none,font=\small,midway,above ] {$\overline{1}$};
\draw (x)--(o) node[draw=none,fill=none,font=\small,midway,above right ] {$\overline{2}$};
\draw (x)--(p) node[draw=none,fill=none,font=\small,midway,below ] {$\overline{3}$};
\draw (y)--(n) node[draw=none,fill=none,font=\small,midway,above ] {$\overline{4}$};
\draw (y)--(o) node[draw=none,fill=none,font=\small,midway,above left ] {$\overline{5}$};
\draw (y)--(p) node[draw=none,fill=none,font=\small,midway,below ] {$\overline{6}$};
\draw (m)--(x) node[draw=none,fill=none,font=\small,midway,left ] {$\overline{a}$};
\end{tikzpicture}
\caption{A subgraph for Case III.2 of Theorem~4.2. \label{2Dim-CaseIII2}}
\end{center}
\end{figure}

\begin{figure}[p]
\begin{center}
\begin{tikzpicture}
\node[vertex] (x) at (0,1.5) [label={left:$x$}]{};
\node[vertex] (y) at (4,1.5) [label={right:$y$}]{};
\node[vertex] (z) at (3,4) {};
\node[vertex] (m) at (1,4) {};
\node[vertex] (n) at (2,3) {};
\node[vertex] (o) at (2,1.5) {};
\node[vertex] (p) at (2,0) {};
\draw (m)--(z) node[draw=none,fill=none,font=\small,midway,above ] {$\overline{v}$};
\draw (x)--(n) node[draw=none,fill=none,font=\small,midway,above ] {$\overline{1}$};
\draw (x)--(o) node[draw=none,fill=none,font=\small,midway,above right ] {$\overline{2}$};
\draw (x)--(p) node[draw=none,fill=none,font=\small,midway,below ] {$\overline{3}$};
\draw (y)--(n) node[draw=none,fill=none,font=\small,midway,above ] {$\overline{4}$};
\draw (y)--(o) node[draw=none,fill=none,font=\small,midway,above left ] {$\overline{5}$};
\draw (y)--(p) node[draw=none,fill=none,font=\small,midway,below ] {$\overline{6}$};
\draw (m)--(x) node[draw=none,fill=none,font=\small,midway,left ] {$\overline{a}$};
\draw (z)--(n) node[draw=none,fill=none,font=\small,midway,above ] {$\overline{b}$};
\draw (o)--(2.5,2.75);
\draw (2.5,2.75)--(z) node[draw=none,fill=none,font=\small,midway,right ] {$\overline{c}$};
\end{tikzpicture}
\caption{A subgraph for Case III.2 of Theorem~4.2. \label{2Dim-CaseIII2a}}
\end{center}
\end{figure}

\begin{figure}[p]
\begin{center}
\begin{tikzpicture}
\node[vertex] (x) at (0,1.5) [label={left:$x$}]{};
\node[vertex] (y) at (4,1.5) [label={right:$y$}]{};
\node[vertex] (z) at (3.5,4.3) {};
\node[vertex] (m) at (1.3,4.3) {};
\node[vertex] (n) at (2,3) {};
\node[vertex] (o) at (2,1.5) {};
\node[vertex] (p) at (2,0) {};
\draw (m)--(z) node[draw=none,fill=none,font=\small,midway,above ] {$\overline{v}$};
\draw (x)--(n) node[draw=none,fill=none,font=\small,midway,above ] {$\overline{1}$};
\draw (x)--(o) node[draw=none,fill=none,font=\small,midway,above right ] {$\overline{2}$};
\draw (x)--(p) node[draw=none,fill=none,font=\small,midway,below ] {$\overline{3}$};
\draw (n)--(y) node[draw=none,fill=none,font=\small,midway,below ] {$\overline{4}$};
\draw (y)--(o) node[draw=none,fill=none,font=\small,midway,below ] {$\overline{5}$};
\draw (y)--(p) node[draw=none,fill=none,font=\small,midway,below ] {$\overline{6}$};
\draw (m)--(x) node[draw=none,fill=none,font=\small,midway,left ] {$\overline{a}$};
\draw (z)--(n) node[draw=none,fill=none,font=\small,midway,above ] {$\overline{b}$};
\draw (z)--(o) node[draw=none,fill=none,font=\small,near end,left ] {$\overline{c}$};
\draw (z)--(p) node[draw=none,fill=none,font=\small,near end,left ] {$\overline{d}$};
\draw (m)--(y) node[draw=none,fill=none,font=\small,near start,above ] {$\overline{e}$};
\end{tikzpicture}
\caption{A subgraph for Case III.2 of Theorem~4.2. \label{2Dim-CaseIII2b}}
\end{center}
\end{figure}

Observing the subgraph of $G$ in Figure~\ref{2Dim-CaseIII2b}, we see that $G$ must contain a copy of $K_{4,3}$. 
Since $M(K_{4,3})$ is a triangulation of $T^2$ by Proposition \ref{TorusMatchingProp}, and $T^2$ cannot be a proper submanifold of a connected homology manifold without boundary, $G = K_{4,3}$.   This completes the proof.
\end{proof}

Although the appearance of the torus in Theorem \ref{2dimMan} might lead us to believe that manifold matching complexes are more plentiful in higher dimension, the reverse is actually true. If $d \ge 3,$ then the following theorem shows that the only manifold matching complexes without boundary are spheres.
The graphs are constructed from the set of basic sphere graphs $\mathcal{SG}$ listed in
Equation~\eqref{BasicSphereGraphs}.

\begin{theorem}\label{ManifoldThm} 
Let $G$ be a simple graph such that $M(G)$ is a $d$-dimensional homology  manifold without boundary for some $d \ge 3$.  Then $G$ is the disjoint union of copies of $P_3$, $C_5$, and $K_{3,2}$, and
therefore $M(G)$ is a combinatorial $d$-sphere.
\end{theorem}

\begin{proof}
We will prove that $G$ is disconnected and therefore
$M(G) = M(G_1) \ast M(G_2)$ by Lemma~\ref{DisjointJoinLemma}.
This will imply that $M$ is a combinatorial sphere.

Let $\dim M = d \ge 3$. 
Assume, by way of induction, that if $G'$ is a simple graph with $M(G')$ a homology
$(d-1)$-sphere, then $G'$ is the disjoint union of copies of
the basic sphere graphs in \eqref{BasicSphereGraphs}.
This is already known to be true when $d=3$ by Theorem~\ref{2-sphere}.

Let $v$ be a vertex of $M$. Since $M$ is a homology manifold, $M(G_{\overline{v}}) = \link_{M} v$ is a $(d-1)$-sphere.
By assumption, $G_{\overline{v}}=H \sqcup J$, where $H$ is one of the basic sphere
graphs and $J$  contains at least one of $2P_3$, $C_5$ or $K_{3,2}$.
We will consider each case and show that $G$ must be disconnected in each case.
Therefore $M$ is a sphere.

Now assume that $G$ is connected.
Notice that $G_{\overline{v}}$ cannot contain a connected component that has more than $6$ edges, since $G_{\overline{v}}$ is the disjoint union of copies of elements of $\mathcal{SG}$. 
We will use this fact to reach a contradiction in each of the following cases.

\textbf{Case I}. $G_{\overline{v}}$ contains three disjoint copies of $P_3$. We will label the edges as in Figure~\ref{Manifold-CaseI}.

\begin{figure}[h]
\begin{center}
\begin{tikzpicture}
\node[vertex] (a) at (0,0) {};
\node[vertex] (b) at (0,1) {};
\node[vertex] (c) at (0,2) {};
\node[vertex] (d) at (1,3) {};
\node[vertex] (e) at (2,3) {};
\node[vertex] (f) at (1.5,2) {};
\node[vertex] (g) at (1.5,1) {};
\node[vertex] (h) at (1.5,0) {};
\node[vertex] (i) at (3,2) {};
\node[vertex] (j) at (3,1) {};
\node[vertex] (k) at (3,0) {};
\draw (a)--(b) node[draw=none,fill=none,font=\small,midway,left] {$\overline{2}$};
\draw (b)--(c) node[draw=none,fill=none,font=\small,midway,left] {$\overline{1}$};
\draw (d)--(e) node[draw=none,fill=none,font=\small,midway,above] {$\overline{v}$};
\draw (h)--(g) node[draw=none,fill=none,font=\small,midway,right] {$\overline{4}$};
\draw (g)--(f) node[draw=none,fill=none,font=\small,midway,right] {$\overline{3}$};
\draw (i)--(j) node[draw=none,fill=none,font=\small,midway,right] {$\overline{5}$};
\draw (j)--(k) node[draw=none,fill=none,font=\small,midway,right] {$\overline{6}$};
\end{tikzpicture}
\caption{A subgraph for Case I of Theorem~4.3. \label{Manifold-CaseI}}
\end{center}
\end{figure}
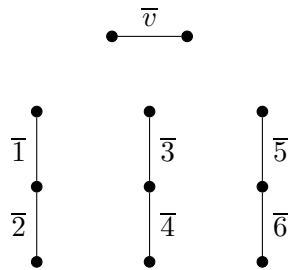

Since $G$ is connected, there must be edges $\overline{a}$ and $\overline{b}$ connecting $\overline{v}$ with the middle and right $P_3$, respectively, in Figure~\ref{Manifold-CaseI}. Then $G_{\overline{1}}$ will contain a connected subgraph with edges $\overline{3}$, $\overline{4}$, $\overline{5}$, $\overline{6}$, $\overline{v}$, $\overline{a}$, and $\overline{b}$.
But $G_{\overline{1}}$ cannot contain a connected subgraph with more than $6$ edges, so this is a contradiction.

\textbf{Case II}.  $G_{\overline{v}}=H\sqcup J$, where $J$ contains at least
one of $C_5$ or $K_{3,2}$.  Note that $J$ has at least 5 edges.
Let $\overline{1}$ be an edge of $H$. Then $G_{\overline{1}}$ contains all edges of $J$,
$\overline{v}$, and at least one edge connecting $\overline{v}$ with
$C_5$ or $K_{3,2}$ from $J$.
Again, we have a contradiction, since
$G_{\overline{1}}$ cannot contain a connected subgraph with more than $6$ edges.

Thus, in both cases a contradiction shows the graph $G$ is disconnected;
write $G=G_1\sqcup G_2$ for nonempty subgraphs $G_1$ and $G_2$.
Then $M(G)=M(G_1)\ast M(G_2)$, and $M(G_1)$, $M(G_2)$ and $M(G)$ are
all spheres by Proposition~\ref{JoinManifold}. Then by the 
induction assumption $G_1$ and $G_2$, and hence $G$, are each disjoint
copies of the basic sphere graphs. Thus, by Proposition~\ref{SphereMatchingProp} $M(G)$ is a combinatorial sphere.
\end{proof}

Thus we see that the only homology manifolds without boundary that occur as
matching complexes are the 2-dimensional torus and combinatorial spheres of all dimensions,
given by Propositions~\ref{TorusMatchingProp} and \ref{SphereMatchingProp} respectively.
Note that the spheres that are matching complexes are all shellable, as they are the joins of shellable spheres.
Also, all of these matching complexes of dimension $d$ have at most 
$3d+3$ vertices, except for the 2-dimensional matching complex of $K_{4,3}$,
which has $12 = 3(2)+6$ vertices.

\section{Manifolds with boundary}\label{Manifolds_w/_Boundary Section}

We turn now to the question of which homology manifolds with boundary 
(of dimension at least 1)  are matching complexes.
Here we find once again that the dimension 2 case has the most
complicated answer.
By Corollary~\ref{disc-cor}, the only disconnected manifolds with boundary
(of dimension at least 1)
that are matching complexes are the matching complex of $C_4$, which is
$2P_2$, and the matching complex of $K_4$, which is $3P_2$.
So in what follows we assume the manifold is connected.
A connected $1$-dimensional  manifold with boundary is just a path.
The nontrivial paths that arise as matching complexes have 2 to 5 vertices;
see Table~\ref{PathMatchingComplexes}.

Next we look at matching complexes of disconnected graphs.
We start by applying Proposition~\ref{JoinManifold} to matching
complexes of disconnected graphs.

\begin{corollary}\label{BallCor}
Let $G$ be a disconnected graph such that $M(G)$ is a
$d$-dimensional manifold with boundary.
Then for some graphs $G_1$ and $G_2$,
$G=G_1\sqcup G_2$,
$M(G)=M(G_1)\ast M(G_2)$, and for some $k$,
$M(G_1)$ is a  
$k$-sphere or $k$-ball, and $M(G_2)$ is a $(d-k-1)$-ball.
Thus $M(G)$ is a ball.

\end{corollary}

From this and the classification of 0- and 1-dimensional matching complexes
that are spheres and balls,  we can find all disconnected graphs that have as
their matching complex a 2-dimensional manifold
with boundary. 

\begin{theorem}
Let  $G$ be a disconnected graph such that $M=M(G)$ is a 2-dimensional
manifold with boundary.  
Then $M$ is a ball, and
$|M|\cong B^0\ast B^1$, $|M|\cong B^0\ast S^1$, or
$|M|\cong S^0\ast B^1$.
Furthermore the pair $G,M(G)$ is one of the pairs in the following
table.

\begin{center}
\begin{tabular}{c|c|l}
{\em Graph $G$} & {\em Manifold $M(G)$} & {\em Description of $M(G)$} \\ \hline
$3P_2$ & $P_1 \ast P_2$ & {\em Triangle} \\
$2P_2 \sqcup P_3$ & $P_1 \ast P_3$ & {\em Two triangles sharing an edge} \\
$P_2 \sqcup P_5$ & $P_1 \ast P_4$ & {\em Chain of three triangles sharing a vertex} \\
$P_2 \sqcup \Gamma$ & $P_1 \ast P_5$ & {\em Chain of four triangles sharing a vertex} \\
$P_2 \sqcup 2P_3$ & $P_1 \ast C_4$ & {\em Triangulated square} \\
$P_2 \sqcup C_5$ & $P_1 \ast C_5$ & {\em Triangulated pentagon} \\
$P_2 \sqcup K_{3,2}$ & $P_1 \ast C_6$ & {\em Triangulated hexagon} \\
$P_3 \sqcup P_5$ &  $2P_1\ast P_4$ & {\em Six triangles: suspension over
path of three edges} \\
$P_3 \sqcup \Gamma$ &  $2P_1\ast P_5$ & {\em Eight triangles: suspension over
path of four edges}
\end{tabular}
\end{center}

\end{theorem}

A surprising variety of 2-dimensional manifolds arise as matching complexes of
connected graphs.

\begin{theorem}\label{2d_with_boundary}
Let $G$ be a connected graph such that $M=M(G)$ is a 2-dimensional
manifold with boundary.  
Then $M$ is one of the following four topological types, arising from
the following graphs.
\begin{enumerate}
\item $M$ is a ball.  $G$ is the spider graph $\mbox{Sp}_3$.
\item $M$ is a triangulated annulus.  $G$ is the following graph with 7 vertices
      and 8 edges:

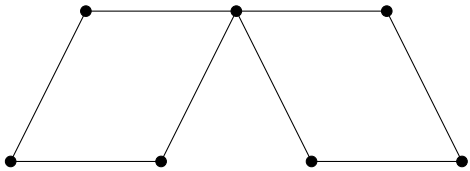
\vspace{6pt}

\begin{tikzpicture}
\draw (3,2)--(1,2)--(0,0)--(2,0)--(3,2)--(5,2)--(6,0)--(4,0)--(3,2);
\filldraw[black] (0,0) circle (2pt);
\filldraw[black] (2,0) circle (2pt);
\filldraw[black] (4,0) circle (2pt);
\filldraw[black] (6,0) circle (2pt);
\filldraw[black] (1,2) circle (2pt);
\filldraw[black] (3,2) circle (2pt);
\filldraw[black] (5,2) circle (2pt);
\end{tikzpicture}


\item $M$ is a triangulated M\"{o}bius strip.
      \begin{enumerate}
      \item $G= C_7$, the 7-cycle
      \item $G$ is the following graph with 7 vertices and 8 edges:

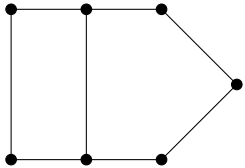
\begin{tikzpicture}
\draw (0,0)--(1,0)--(2,0)--(3,1)--(2,2)--(1,2)--(0,2)--(0,0);
\draw (1,0)--(1,2);
\filldraw[black] (0,0) circle (2pt);
\filldraw[black] (1,0) circle (2pt);
\filldraw[black] (2,0) circle (2pt);
\filldraw[black] (0,2) circle (2pt);
\filldraw[black] (1,2) circle (2pt);
\filldraw[black] (2,2) circle (2pt);
\filldraw[black] (3,1) circle (2pt);
\end{tikzpicture}
      \item $G$ is the following graph with 7 vertices and 9 edges:

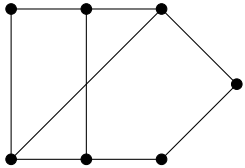
\begin{tikzpicture}
\draw (0,0)--(1,0)--(2,0)--(3,1)--(2,2)--(1,2)--(0,2)--(0,0);
\draw (1,0)--(1,2);
\draw (0,0)--(2,2);
\filldraw[black] (0,0) circle (2pt);
\filldraw[black] (1,0) circle (2pt);
\filldraw[black] (2,0) circle (2pt);
\filldraw[black] (0,2) circle (2pt);
\filldraw[black] (1,2) circle (2pt);
\filldraw[black] (2,2) circle (2pt);
\filldraw[black] (3,1) circle (2pt);
\end{tikzpicture}
      \item $G$ is the following graph with 7 vertices and 10 edges:

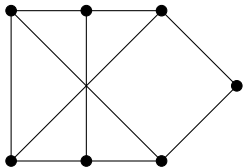
\begin{tikzpicture}
\filldraw[black] (0,0) circle (2pt);
\filldraw[black] (1,0) circle (2pt);
\filldraw[black] (2,0) circle (2pt);
\filldraw[black] (3,1) circle (2pt);
\filldraw[black] (0,2) circle (2pt);
\filldraw[black] (1,2) circle (2pt);
\filldraw[black] (2,2) circle (2pt);
\draw (0,0)--(1,0)--(2,0)--(3,1)--(2,2)--((1,2)--(0,2)--(2,0);
\draw (0,2)--(0,0)--(2,2);
\draw (1,0)--(1,2);
\end{tikzpicture}

      \end{enumerate}

\item $M$ is a triangulated torus with a 2-ball removed.

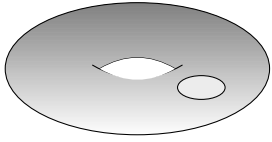
\begin{tikzpicture}[scale = 0.5]
\shadedraw [rotate=0] (0,0) ellipse (100pt and 50pt);
\draw (-1,0) to[bend left] (1,0);
\draw (-1.2,.1) to[bend right] (1.2,.1);
\path [fill=white] (-1,0) to[bend left] (1,0) to[bend left] (-1,0);
\filldraw [gray!15] (1.7,-.5) ellipse (18pt and 9pt);
\draw (1.7,-.5) ellipse (18pt and 9pt);
\end{tikzpicture}

      \begin{enumerate}

      \item $G$ is the following graph with 7 vertices and 9 edges:

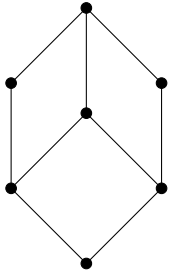
\begin{tikzpicture}[scale=0.5]
\draw (0,4.8)--(0,2)--(2,0)--(4,2)--(4,4.8)--(2,6.8)--(2,4)--(4,2);
\draw (0,2)--(2,4);
\draw (0,4.8)--(2,6.8);
\filldraw[black] (0,2) circle (4pt);
\filldraw[black] (2,0) circle (4pt);
\filldraw[black] (4,2) circle (4pt);
\filldraw[black] (0,4.8) circle (4pt);
\filldraw[black] (2,4) circle (4pt);
\filldraw[black] (4,4.8) circle (4pt);
\filldraw[black] (2,6.8) circle (4pt);
\end{tikzpicture}


      \item $G$ is the following graph with 7 vertices and 10 edges:

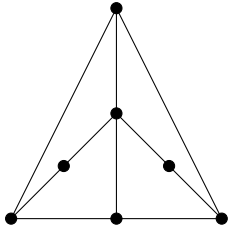
\begin{tikzpicture}[scale=0.7]
\draw (0,0)--(2,0)--(4,0)--(3,1)--(2,2)--(1,1)--(0,0);
\draw (0,0)--(2,4)--(4,0);
\draw (2,0)--(2,2)--(2,4);
\filldraw[black] (0,0) circle (3pt);
\filldraw[black] (2,0) circle (3pt);
\filldraw[black] (4,0) circle (3pt);
\filldraw[black] (1,1) circle (3pt);
\filldraw[black] (3,1) circle (3pt);
\filldraw[black] (2,2) circle (3pt);
\filldraw[black] (2,4) circle (3pt);
\end{tikzpicture}
      \item $G$ is the following graph with 7 vertices and 11 edges:

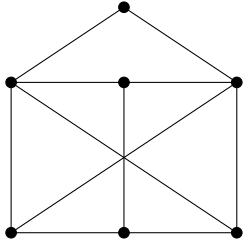
\begin{tikzpicture}
\filldraw[black] (0,0) circle (2pt);
\filldraw[black] (1.5,0) circle (2pt);
\filldraw[black] (3,0) circle (2pt);
\filldraw[black] (0,2) circle (2pt);
\filldraw[black] (1.5,2) circle (2pt);
\filldraw[black] (3,2) circle (2pt);
\filldraw[black] (1.5,3) circle (2pt);
\draw (0,0)--(1.5,0)--(3,0)--(3,2)--(1.5,2)--(0,2)--(0,0);
\draw (3,0)--(0,2)--(1.5,3)--(3,2)--(0,0);
\draw (1.5,0)--(1.5,2);
\end{tikzpicture}

     \end{enumerate}
\end{enumerate}
\end{theorem}

\begin{proof}
The proof relies heavily on analysis of the links of vertices in the
manifold.  We start by reviewing the possible structures of these links.
We showed previously (Lemma~\ref{LinkLemma}) that the
link of a vertex $v$ of $M$ is an induced subcomplex of $M$,  which
is then the matching complex of the subgraph $G_{\overline{v}}$ of $G$.
We write $x_v$ and $y_v$ for the vertices
of the edge $\overline{v}$ of $G$.

Assume $G$ is a connected graph and $M=M(G)$ is a 2-dimensional manifold
with boundary.

If $v$ is a boundary vertex of $M$, then the link of $v$ is a
1-dimensional ball, that is, a path $P_j$, $2\le j\le 5$, so the
corresponding subgraph $G_{\overline{v}}$ of $G$ is in the set $\{2P_2,P_3\sqcup P_2,
P_5,\Gamma\}$.  Also, the endpoints of $P_j$ are also boundary vertices of
$M$.

If $v$ is an interior vertex of $M$, then the link of $v$ is a
1-dimensional sphere, that is, a cycle $C_j$, $4\le j\le 6$, so the
corresponding subgraph $G_{\overline{v}}$ of $G$ is in the set
$\{2P_3,C_5,K_{3,2}\}$.

Note that in all cases $G_{\overline{v}}$ does not contain a triangle.
In fact, it is easy to see that this implies that the entire graph $G$ does
not contain a triangle: if an edge $\overline{v}$ were disjoint from the triangle,
$G_{\overline{v}}$ would contain a triangle; otherwise all edges of $G$  would
contain a vertex of the triangle, and for any such edge $\overline{v}$, $G_{\overline{v}}$
would not be in either of the sets above.

We split the proof into cases based on the structure of the links of
boundary vertices.

\textbf{Case I}.  For some boundary vertex $v$, $\link_{M}(v) = P_2$,
so $G_{\overline{v}}= 2P_2$.  Let $1$ and $2$ be the neighbors of $v$ in
$M$.
Thus $G_{\overline{v}}$ has the two nonincident edges $\overline{1}$ and $\overline{2}$,
which are also boundary vertices, and all other edges of $G$ are incident
to $\overline{v}$.
Since $G$ is connected,
$G$  must have an edge connecting $\overline{v}$ and $\overline{1}$,
say (without loss of generality) $\overline{3} = \{x_v,x_1\}$.
Since $G_{\overline{2}}$
does not contain a triangle, $G$ does not contain edges
$\{x_v,y_1\}$ and $\{y_v,x_1\}$.
We split Case I into two subcases, depending on whether the edge
$\{y_v,y_1\}$ is in $G$.

\textbf{Case I.1}.  (See Figure~\ref{CaseI1}.) If $\{y_v,y_1\}$ is in $G$, then
$G_{\overline{2}}$ contains a 4-cycle, so must be $\Gamma$ (since $2$ is also
a boundary vertex), and all edges of $G$ besides $\overline{2}$ and the edges
of $\Gamma$ are edges with one endpoint in $\overline{v}$ and one endpoint in
$\overline{2}$. By connectivity, there must be at least one such edge. Say
without loss of generality it is the edge $\{y_2,y_v\}$.
The pendant edge of $\Gamma$ is either $\{x_v,z\}$ or $\{y_v,z\}$, as
shown in Figure~\ref{CaseI1}.
Since $G_{\overline{1}}$ does not contain a triangle, $G$ does not contain
edges $\{y_2,x_v\}$ or $\{x_2,y_v\}$.
The graph $G$ may or may not contain the edge $\{x_2,x_v\}$.  In all
cases $G_{\overline{3}}$ is either $P_4$ or a $P_4$ with an additional edge on one
of its interior vertices; these are not possible.  

\begin{figure}[p]
\begin{center}
\begin{tikzpicture}
\draw (0,0)--(2,0) node[draw=none,fill=none,font=\footnotesize,midway,below] {$\overline{1}$};
\draw (0,0)--(0,3) node[draw=none,fill=none,font=\footnotesize,midway,left] {$\overline{3}$};
\draw (0,3)--(2,3) node[draw=none,fill=none,font=\footnotesize,midway,above] {$\overline{v}$};
\draw (4,0)--(4,3) node[draw=none,fill=none,font=\footnotesize,midway,right] {$\overline{2}$};
\draw (4,0) node[below right] {$x_2$}--
      (4,3) node[above right] {$y_2$}--
      (2,3) node[above] {$y_v$}--
      (0,3) node[above left] {$x_v$}--
      (0,0) node[below left] {$x_1$}--
      (2,0) node[below] {$y_1$}--(2,3);
\draw [dashed] (0,3)--(4,0);
\draw [dashed] (0,3)--(1,4) node[above] {$z$}--(2,3);
\filldraw[black] (0,0) circle (2pt);
\filldraw[black] (2,0) circle (2pt);
\filldraw[black] (4,0) circle (2pt);
\filldraw[black] (0,3) circle (2pt);
\filldraw[black] (2,3) circle (2pt);
\filldraw[black] (4,3) circle (2pt);
\filldraw[black] (1,4) circle (2pt);
\end{tikzpicture}
\caption{Graph for Case I.1; $G$ contains exactly one of the edges from $z$ and
may or may not contain the edge $\{x_v,x_2\}$. \label{CaseI1}}
\end{center}
\end{figure}
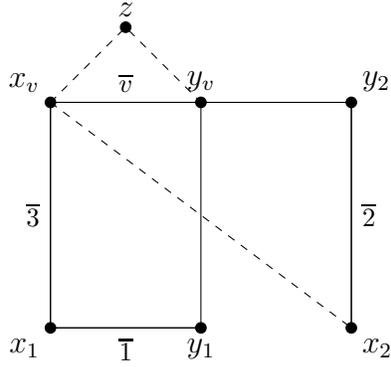

\textbf{Case I.2}. (See Figure~\ref{CaseI2}.)  
If $\{y_v,y_1\}$ is not in $G$, then
all remaining edges of $G$ are incident to $\overline{v}$, but not to $\overline{1}$.
(Otherwise $G_{\overline{2}}$ would have a triangle, as noted above.)
The subgraph $G_{\overline{2}}$ contains the path
$P_4: y_1, x_1, x_v, y_v$.
Since 2 is a boundary vertex of $M$, its link is a ball, so
$G_{\overline{2}}$ is $P_5$ or $\Gamma$.
With no more edges incident to $\overline{1}$, $\Gamma$ is not possible,
so $G_{\overline{2}}$ is $P_5$.
That is, there is exactly one more
edge in $G$ not incident to $\overline{2}$, call it $\overline{4}$,
and it contains the vertex $y_v$.
Now turn again to $G_{\overline{1}}$.  It contains edges $\overline{v}$, $\overline{4}$,
$\overline{2}$, and any other edges (at least one) incident to $\overline{v}$ and
$\overline{2}$.
Since 1 is a boundary vertex of $M$, $G_{\overline{1}}$ must also be
$P_5$ or $\Gamma$.
If it is $\Gamma$, then $G$ contains exactly one edge $\overline{5}$
containing a vertex of $\overline{2}$ and the vertex $y_{\overline{v}}$.  But then
the graph $G_{\overline{3}}$ is $P_4$, which is not possible.
This leaves only the possibility that $G_{\overline{1}}$ is $P_5$, and the
graph $G$ is the spider graph $\mbox{Sp}_3$.  
This is part~1 of the Theorem.

\begin{figure}[p]
\begin{center}
\begin{tikzpicture}
\draw (0,0)--(2,0) node[draw=none,fill=none,font=\footnotesize,midway,below] {$\overline{1}$};
\draw (0,0)--(0,3) node[draw=none,fill=none,font=\footnotesize,midway,left] {$\overline{3}$};
\draw (0,3)--(2,3) node[draw=none,fill=none,font=\footnotesize,midway,above] {$\overline{v}$};
\draw (2,3)--(1,4) node[draw=none,fill=none,font=\footnotesize,midway,right] {$\overline{4}$};
\draw (4,0)--(4,3) node[draw=none,fill=none,font=\footnotesize,midway,right] {$\overline{2}$};
\draw (4,0)--(0,3);
\draw (4,0) node[below right] {$x_2$}--
      (4,3) node[above right] {$y_2$};
\draw (2,3) node[right] {$y_v$}--
      (0,3) node[above left] {$x_v$}--
      (0,0) node[below left] {$x_1$}--
      (2,0) node[below] {$y_1$};
\draw (1,4) node[above] {$z$}--(2,3);
\filldraw[black] (0,0) circle (2pt);
\filldraw[black] (2,0) circle (2pt);
\filldraw[black] (4,0) circle (2pt);
\filldraw[black] (0,3) circle (2pt);
\filldraw[black] (2,3) circle (2pt);
\filldraw[black] (4,3) circle (2pt);
\filldraw[black] (1,4) circle (2pt);
\end{tikzpicture}
\caption{Graph for Case I.2 \label{CaseI2}}
\end{center}
\end{figure}
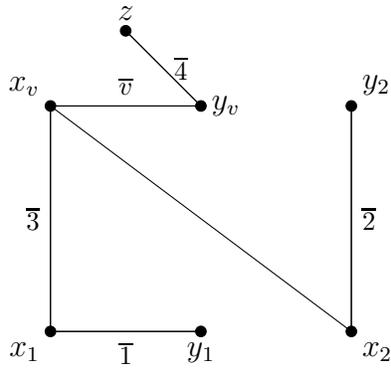

\textbf{Case II}. No boundary vertex has $\link_{M}(v) = P_2$, and for
some boundary vertex $v$, $\link_{M}(v) = P_3$, so
$G_{\overline{v}}=P_3\sqcup P_2$.
Let the link of $v$ have vertices 1, 2, 3 (forming a path in that order).
Then 1 and 3 are boundary vertices, and $G_{\overline{v}}$ consists of the
path with two edges $\overline{1}$ and $\overline{3}$ and the disjoint edge $\overline{2}$.
Write the common vertex in $G$ of $\overline{1}$ and $\overline{3}$ as $y_1$, with
$x_1$ and $x_3$ their other vertices.  Besides $\overline{v}$, $\overline{1}$, $\overline{2}$
and $\overline{3}$, every other edge of $G$ must contain a vertex of $\overline{v}$.
Also, $G$ is connected, so there must be at least one edge between a vertex
of $\overline{v}$ and a vertex of $\overline{2}$.  Without loss of generality, $G$
contains the edge $\overline{4}=\{x_v,x_2\}$.  Since $G$ contains no triangles,
the only other possible edge between $\overline{v}$ and $\overline{2}$ is
$\{y_v,y_2\}$.  We split Case II into two cases, depending on whether that
edge is in $G$.

\textbf{Case II.1}. Assume $\{y_v,y_2\}$ is not an edge of $G$.  Consider
$G_{\overline{3}}$. It contains the path
$P_4: y_v, x_v, x_2, y_2$, so it is either $P_5$ or $\Gamma$.  Without the
edge $\{y_v,y_2\}$, $G_{\overline{3}}$ cannot be $\Gamma$, so $G_{\overline{3}}$ is
$P_5: z, y_v, x_v, x_2, y_2$. All remaining edges of $G$ must connect the
edge $\overline{v}$ with the edge $\overline{3}$.  If $z=x_1$, then $G_{\overline{1}}$
contains the path $P_4: y_v,x_v,x_2,y_2$ and at most one other edge
connecting $\overline{v}$ and $x_3$. The subgraph $G_{\overline{1}}$ must then
be $P_5: x_3,y_v,x_v,x_2,y_2$. Then $G_{\overline{4}}$ contains a cycle of length 4,
and no edge of $G$ can complete it to $\Gamma$.
If $z$ is not in the edge $\overline{1}$, then $G_{\overline{2}}$ contains
$2P_3$ (edges $\overline{1}$, $\overline{3}$, $\overline{v}$ and $\{z,y_v\}$).
But then there can be no edges between vertices of $\overline{v}$ and vertices of
$\overline{1}$ and $\overline{3}$, contradicting the connectedness of $G$.

\textbf{Case II.2}. Assume $\overline{5}=\{y_v,y_2\}$ is an edge of $G$.  Since $G$ does
not contain a triangle, all remaining edges of $G$ contain a vertex of
$\overline{v}$, but no vertex of $\overline{2}$.  Also, the subgraphs $G_{\overline{1}}$
and $G_{\overline{3}}$ each contain the $C_4: x_v, y_v, y_2, x_2, x_v$, so
must be copies of $\Gamma$.  We consider what additional edge or edges
are needed for this.

\textbf{Case II.2.a}. Suppose an additional edge contains (without loss of
generality) the vertex $x_v$ and a vertex $z$ not in $\overline{1}$ or $\overline{3}$.
Then there must be one more edge to make $G$ connected,
and it must contain the vertex $y_1$ common to edges $\overline{1}$
and $\overline{3}$ (so as not to change $G_{\overline{1}}$ or $G_{\overline{3}}$). 
But this would result in an invalid subgraph for $G_{\overline{2}}$,
since a connected component of $G_{\overline{2}}$ would contain six vertices,
but this is impossible for basic sphere graphs.

\textbf{Case II.2.b}. (See Figure~\ref{II2b}.) Suppose there are edges $\overline{6}$ and $\overline{7}$
from the same vertex $x_v$ (without
loss of generality) to vertices $x_1$ and $x_3$.
The result is the union of two 4-cycles with a common vertex,
with matching complex an annulus, a
manifold with boundary.
Any other edge of $G$ must be incident to $\overline{v}$.  In addition, the graph
so far contains copies of $\Gamma$ for $G_{\overline{1}}$ and $G_{\overline{3}}$, so any other edge of $G$ must be incident to
the edges $\overline{v}$, $\overline{1}$, and $\overline{3}$. Since $G$ does not contain a triangle, edge $y_1x_v$ does not exist. Also, $G$ does not contain edge $y_1y_v$ because of $G_{\overline{4}}$.  So Case II.2.b must be the graph and its matching complex
shown in
Figure~\ref{II2b}.  This is part~2 of the Theorem.

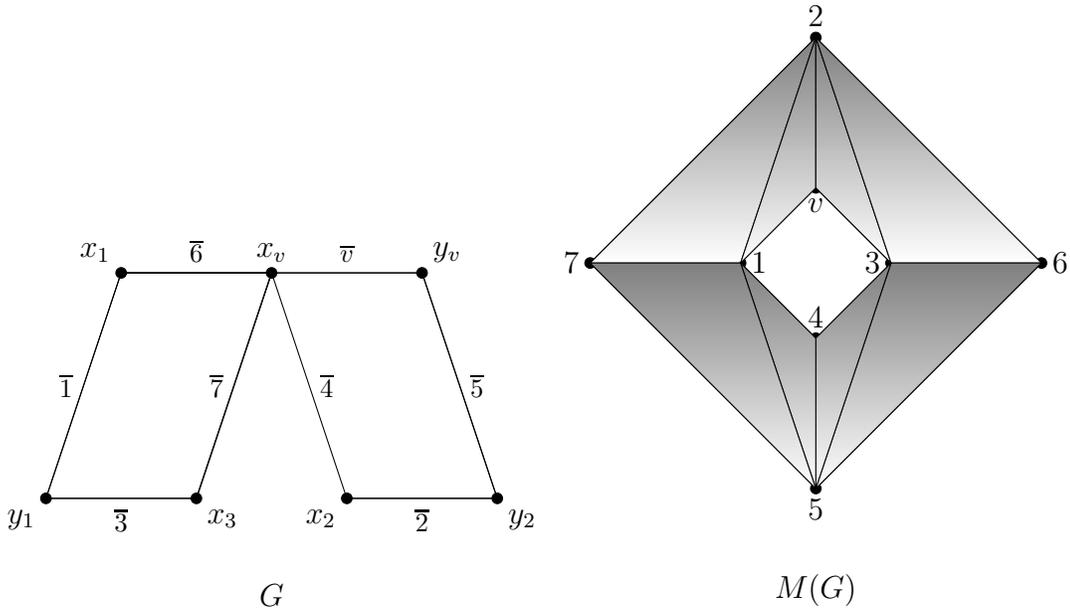
\begin{figure}[h]
\begin{tikzpicture}
\draw (1,3)--(0,0) node[draw=none,fill=none,font=\footnotesize,midway,left] {$\overline{1}$};
\draw (1,3)--(3,3) node[draw=none,fill=none,font=\footnotesize,midway,above] {$\overline{6}$};
\draw (2,0)--(3,3) node[draw=none,fill=none,font=\footnotesize,midway,left] {$\overline{7}$};
\draw (0,0)--(2,0) node[draw=none,fill=none,font=\footnotesize,midway,below] {$\overline{3}$};
\draw (2,0)--(3,3)--(1,3);
\draw (3,3)--(4,0) node[draw=none,fill=none,font=\footnotesize,midway,right] {$\overline{4}$};
\draw (4,0)--(6,0) node[draw=none,fill=none,font=\footnotesize,midway,below] {$\overline{2}$};
\draw (6,0)--(5,3) node[draw=none,fill=none,font=\footnotesize,midway,right] {$\overline{5}$};
\draw (5,3)--(3,3) node[draw=none,fill=none,font=\footnotesize,midway,above] {$\overline{v}$};
\draw (0,0) node[below left] {$y_1$}--
      (2,0) node[below right] {$x_3$}--
      (3,3) node[above] {$x_v$}--
      (1,3) node[above left] {$x_1$}--(0,0);
\draw (4,0) node[below left] {$x_2$}--
      (6,0) node[below right] {$y_2$}--
      (5,3) node[above right] {$y_v$}--(3,3);
\filldraw[black] (0,0) circle (2pt);
\filldraw[black] (2,0) circle (2pt);
\filldraw[black] (4,0) circle (2pt);
\filldraw[black] (6,0) circle (2pt);
\filldraw[black] (1,3) circle (2pt);
\filldraw[black] (3,3) circle (2pt);
\filldraw[black] (5,3) circle (2pt);
\draw (3,-1) node[below] {$G$};
\end{tikzpicture}
\begin{tikzpicture}
\draw (0,3) node[left] {$7$}--
      (3,0) node[below] {$5$}--
      (6,3) node[right] {$6$}--
      (3,6) node[above] {$2$}--(0,3);
\draw (2,3) node[right] {$1$}--
      (3,2) node[above] {$4$}--
      (4,3) node[left] {$3$}--
      (3,4) node[below] {$v$}--(2,3);
\draw (0,3)--(2,3)--(3,6)--(4,3)--(3,0)--(2,3);
\draw (3,4)--(3,6);
\draw (3,0)--(3,2);
\draw (4,3)--(6,3);
\filldraw[black] (0,3) circle (2pt);
\filldraw[black] (3,0) circle (2pt);
\filldraw[black] (6,3) circle (2pt);
\filldraw[black] (3,6) circle (2pt);
\filldraw[black] (2,3) circle (2pt);
\filldraw[black] (3,2) circle (2pt);
\filldraw[black] (4,3) circle (2pt);
\filldraw[black] (3,4) circle (2pt);
\shadedraw (0,3) --  (3,6) -- (2,3) -- (0,3);
\shadedraw (2,3) --  (3,6) -- (3,4) -- (2,3);
\shadedraw (3,4) --  (3,6) -- (4,3) -- (3,4);
\shadedraw (4,3) --  (3,6) -- (6,3) -- (4,3);
\shadedraw (4,3) --  (6,3) -- (3,0) -- (4,3);
\shadedraw (4,3) --  (3,0) -- (3,2) -- (4,3);
\shadedraw (3,2) --  (3,0) -- (2,3) -- (3,2);
\shadedraw (0,3) --  (2,3) -- (3,0) -- (0,3);
\draw (3,-1) node[below] {$M(G)$};
\end{tikzpicture}
\caption{Graph and Matching Complex for Case II.2.b\label{II2b}}
\end{figure}

\begin{figure}[h]
\begin{center}
\begin{tikzpicture}
\draw (0,0)--(1,0) node[draw=none,fill=none,font=\footnotesize,midway,below] {$\overline{5}$};
\draw (1,0)--(2,0) node[draw=none,fill=none,font=\footnotesize,midway,below] {$\overline{7}$};
\draw (2,0)--(3,1) node[draw=none,fill=none,font=\footnotesize,midway,right] {$\overline{3}$};
\draw (3,1)--(2,2) node[draw=none,fill=none,font=\footnotesize,midway,right] {$\overline{1}$};
\draw (2,2)--(1,2) node[draw=none,fill=none,font=\footnotesize,midway,above] {$\overline{6}$};
\draw (1,2)--(0,2) node[draw=none,fill=none,font=\footnotesize,midway,above] {$\overline{4}$};
\draw (0,2)--(0,0) node[draw=none,fill=none,font=\footnotesize,midway,left] {$\overline{2}$};
\draw (1,0)--(1,2) node[draw=none,fill=none,font=\footnotesize,midway,right] {$\overline{v}$};
\draw (0,0) node[below] {$y_2$};
\draw (1,0) node[below] {$y_v$};
\draw (2,0) node[below] {$x_3$};
\draw (0,2) node[above] {$x_2$};
\draw (1,2) node[above] {$x_v$};
\draw (2,2) node[above] {$x_1$};
\draw (3,1) node[right] {$y_1$};
\filldraw[black] (0,0) circle (2pt);
\filldraw[black] (1,0) circle (2pt);
\filldraw[black] (2,0) circle (2pt);
\filldraw[black] (0,2) circle (2pt);
\filldraw[black] (1,2) circle (2pt);
\filldraw[black] (2,2) circle (2pt);
\filldraw[black] (3,1) circle (2pt);
\draw (1,-1) node[below] {$G$};
\end{tikzpicture}
\begin{tikzpicture}[scale=.7]
\shadedraw (1,0)--(2,2)--(0,2)--(1,0);
\shadedraw (1,0)--(3,0)--(2,2)--(1,0);
\shadedraw (3,0)--(4,2)--(2,2)--(3,0);
\shadedraw (3,0)--(5,0)--(4,2)--(3,0);
\shadedraw (5,0)--(6,2)--(4,2)--(5,0);
\shadedraw (5,0)--(7,0)--(6,2)--(5,0);
\shadedraw (7,0)--(8,2)--(6,2)--(7,0);
\shadedraw (2,2)--(4,2)--(4,3.5)--(2,2);
\shadedraw (4,2)--(6,2)--(4,3.5)--(4,2);
\filldraw[black] (1,0) circle (2pt);
\filldraw[black] (3,0) circle (2pt);
\filldraw[black] (5,0) circle (2pt);
\filldraw[black] (7,0) circle (2pt);
\filldraw[black] (0,2) circle (2pt);
\filldraw[black] (2,2) circle (2pt);
\filldraw[black] (4,2) circle (2pt);
\filldraw[black] (6,2) circle (2pt);
\filldraw[black] (8,2) circle (2pt);
\filldraw[black] (4,3.5) circle (2pt);
\draw (1,0) node[below left] {5};
\draw (3,0) node[below] {6};
\draw (5,0) node[below] {7};
\draw (7,0) node[below right] {4};
\draw (0,2) node[above left] {4};
\draw (2,2) node[above] {3};
\draw (4,2) node[above right] {2};
\draw (6,2) node[above] {1};
\draw (8,2) node[above right] {5};
\draw (4,3.5) node[above] {$v$};
\draw[directed] (0,2)--(1,0);
\draw[directed] (7,0)--(8,2);
\draw (4,-1) node[below] {$M(G)$};
\end{tikzpicture}
\caption{Graph and Matching Complex for Case II.2.c\label{II2c}}
\end{center}
\end{figure}
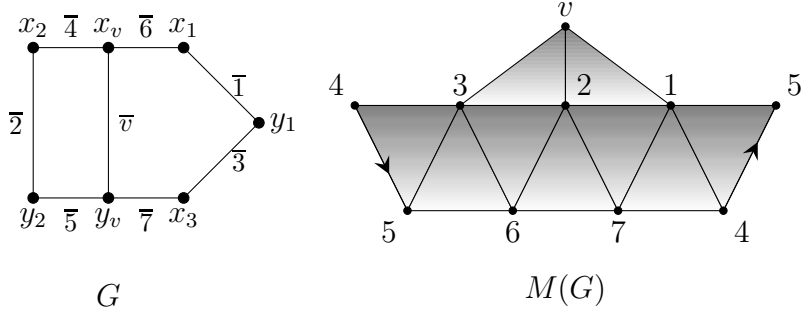

\textbf{Case II.2.c}. Otherwise, the graphs $G_{\overline{1}}$ and $G_{\overline{3}}$
are copies of $\Gamma$ with pendant edges containing different vertices of
$\overline{v}$.  Without loss of generality, the graph $G$ contains
edges $\overline{6}= \{x_v,x_1\}$ and $\overline{7}=\{y_v,x_3\}$.
Any other edge of $G$ would create a triangle.  So the graph $G$
and its matching complex are shown in
Figure~\ref{II2c}.  In this, and in subsequent drawings of
matching complexes arrows show the identification of edges.
This is part~3(b) of the Theorem.

\textbf{Case III}.
No boundary vertex has $\link_{M}(v)= P_2$ or $P_3$, and for some
boundary vertex $v$, $\mbox{link}_{M}(v) = P_4$, so $G_{\overline{v}}=P_5$.
Let the link of $v$ have vertices 1, 2, 3, 4 (forming a path in that order).
Then 1 and 4 are boundary vertices of $M$, $G_{\overline{1}}$ and $G_{\overline{4}}$
are each either $P_5$ or $\Gamma$,  and $G_{\overline{v}}$ consists
of the path with consecutive edges $\overline{2}$, $\overline{4}$, $\overline{1}$,
$\overline{3}$ and vertices $x_2, y_2, x_1, y_1, y_3$. Besides the edges
$\overline{v}$, $\overline{1}$, $\overline{2}$, $\overline{3}$, and $\overline{4}$, every other edge of
$G$ must contain a vertex of $\overline{v}$.
The vertex 2 (and similarly 3) could be either a boundary vertex or an
interior vertex of~$M$.

\textbf{Case III.1}.
Assume 2 is a boundary vertex of $M$. 
(The case where $3$ is a boundary vertex
follows by symmetry.) Then $G_{\overline{2}}= P_5$ or~$\Gamma$.

\textbf{Case III.1.a}.  Assume $G_{\overline{2}}= P_5$.
Since $G_{\overline{2}}$ contains edges $\overline{v}$, $\overline{1}$ and $\overline{3}$, it
contains either $\{x_v,x_1\}$ or $\{x_v,y_3\}$ ($x_v$ being either vertex
of $\overline{v}$), and all other edges of $G$ contain a vertex of $\overline{v}$
and a vertex of $\overline{2}$.
If $G_{\overline{2}}$ contained $\{x_v,x_1\}$, then $G_{\overline{4}}$ would contain
edge $\overline{3}$ and no other edges incident to $\overline{3}$; this cannot happen in
$P_5$ or $\Gamma$.  Thus
$G_{\overline{2}}$ contains $\overline{5}=\{x_v,y_3\}$.  Now consider $G_{\overline{4}}$; it
contains the path $P_4: \overline{v}, \overline{5}, \overline{3}$, and it must be $P_5$
or $\Gamma$.
Since all other edges of $G$ contain a vertex of $\overline{v}$ and a vertex of
$\overline{2}$, $G_{\overline{4}}$ must contain the edge $\overline{6}=\{y_v,x_2\}$.
Thus $G$ contains the 7-cycle
$C_7: y_v,x_v,y_3,y_1,x_1,y_2,x_2,y_v$.
Since $G_{\overline{3}}$ cannot contain a triangle, the only other  possible edge
in $G$ is $\overline{7}=\{x_v,y_2\}$.  However if edge $\overline{7}$ is in $G$, then
the corresponding subgraph $G_{\overline{7}}$ would be $P_2\sqcup P_3$,
so $M$ would
fall into Case II.  So $G$ must be $C_7$.  Then $M$ is a triangulated
M\"{o}bius strip with triangles $215$, $156$, $564$, $643$, $43v$, $3v2$, $v21$.
This is part~3(a) of the Theorem.

\textbf{Case III.1.b}.  Assume $G_{\overline{2}}= \Gamma$.  Since all remaining edges
of $G$ contain a vertex of $\overline{v}$, $G_{\overline{2}}$ is either the 4-cycle
$C_4: x_1, x_v, y_v, y_1, x_1$ with pendant edge $\{y_1,y_3\}$ or the 4-cycle
$C_4: y_1, x_v, y_v, y_3, y_1$ with pendant edge $\{x_1,y_1\}$.
The first alternative is not possible, because it would leave $G_{\overline{1}}$
with only four vertices ($x_2, y_2, x_v, y_v$).  So assume $G$
contains edges $\overline{5}= \{y_1,x_v\}$ and $\overline{6}=\{y_3,y_v\}$.
All remaining edges of $G$ must contain a vertex of $\overline{v}$ and a vertex of
$\overline{2}$.
The 4-cycle with edges $\overline{3},\overline{5}, \overline{v}, \overline{6}$ is also in
$G_{\overline{4}}$, so $G_{\overline{4}}$ must also be $\Gamma$, and $G$ must contain
an edge connecting $x_2$ to $\overline{v}$.
Since $G_{\overline{5}}$ contains edges $\overline{2}$, $\overline{4}$, and $\overline{6}$ and,
being in Case III, $G_{\overline{5}}$ cannot be $P_3\sqcup P_2$, there must be
another edge containing $y_v$ and completing a $P_5$, namely, the edge
$\overline{7}=\{x_2,y_v\}$.

Finally, $G_{\overline{1}}$ contains edges
$\overline{2}$, $\overline{6}$, $\overline{7}$, and $\overline{v}$ and must be $\Gamma$ with
edge $\overline{8}=\{y_2,x_v\}$.
Every other edge of $G$ must be incident to $\overline{v}$ and $\overline{2}$,
but such an edge would create a triangle in $G$.  So the graph $G$ has
just the nine edges described.
The matching complex $M$ is a manifold with triangles
$745$, $456$, $562$, $612$, $12v$, $2v3$, $v34$, $34v$, $618$, $187$, and
$873$. See Figure~\ref{III1b}.
We observe that after gluing together the identified edges, we obtain a $2$-manifold with boundary homeomorphic to the torus with a $2$-ball removed from the surface.
(The 2-ball is bounded by a cycle $1,7,5,2,3,8,6,4,v,1$.)
This is part~4(a) of the Theorem.

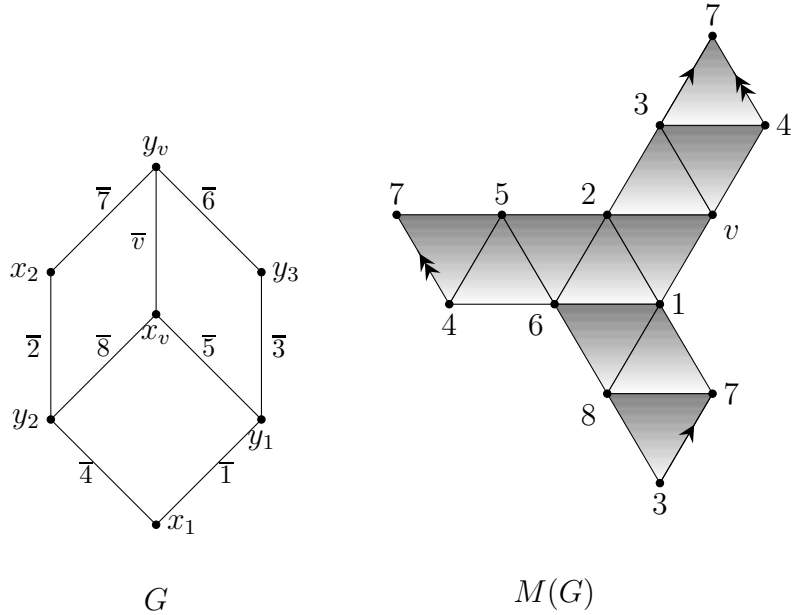
\begin{figure}[h]
\begin{center}
\begin{tikzpicture}[scale=.7]
\draw (0,4.8)--(0,2) node[draw=none,fill=none,font=\footnotesize,midway,left] {$\overline{2}$};
\draw (0,2)--(2,0) node[draw=none,fill=none,font=\footnotesize,midway,left] {$\overline{4}$};
\draw (2,0)--(4,2) node[draw=none,fill=none,font=\footnotesize,midway,right] {$\overline{1}$};
\draw (4,2)--(4,4.8) node[draw=none,fill=none,font=\footnotesize,midway,right] {$\overline{3}$};
\draw (4,4.8)--(2,6.8) node[draw=none,fill=none,font=\footnotesize,midway,above] {$\overline{6}$};
\draw (2,6.8)--(2,4) node[draw=none,fill=none,font=\footnotesize,midway,left] {$\overline{v}$};
\draw (2,4)--(4,2) node[draw=none,fill=none,font=\footnotesize,midway,above] {$\overline{5}$};
\draw (0,2)--(2,4) node[draw=none,fill=none,font=\footnotesize,midway,above] {$\overline{8}$};
\draw (0,4.8)--(2,6.8) node[draw=none,fill=none,font=\footnotesize,midway,above] {$\overline{7}$};
\draw (0,4.8) node[left] {$x_2$};
\draw (0,2) node[left] {$y_2$};
\draw (2,0) node[right] {$x_1$};
\draw (4,2) node[below] {$y_1$};
\draw (4,4.8) node[right] {$y_3$};
\draw (2,6.8) node[above] {$y_v$};
\draw (2,4) node[below] {$x_v$};
\filldraw[black] (0,2) circle (2pt);
\filldraw[black] (2,0) circle (2pt);
\filldraw[black] (4,2) circle (2pt);
\filldraw[black] (0,4.8) circle (2pt);
\filldraw[black] (2,4) circle (2pt);
\filldraw[black] (4,4.8) circle (2pt);
\filldraw[black] (2,6.8) circle (2pt);
\draw (2,-1) node[below] {$G$};
\end{tikzpicture}
\hspace*{18pt}
\begin{tikzpicture}[scale=.7]
\shadedraw (1,0)--(2,1.7)--(0,1.7)--(1,0);
\shadedraw (1,0)--(3,0)--(2,1.7)--(1,0);;
\shadedraw (3,0)--(4,1.7)--(2,1.7)--(3,0);
\shadedraw (3,0)--(5,0)--(4,1.7)--(3,0);
\shadedraw (5,0)--(6,1.7)--(4,1.7)--(5,0);
\shadedraw (4,1.7)--(6,1.7)--(5,3.4)--(4,1.7);
\shadedraw (6,1.7)--(7,3.4)--(5,3.4)--(6,1.7);
\shadedraw (5,3.4)--(7,3.4)--(6,5.1)--(5,3.4);
\shadedraw (4,-1.7)--(5,0)--(3,0)--(4,-1.7);
\shadedraw (4,-1.7)--(6,-1.7)--(5,0)--(4,-1.7);
\shadedraw (4,-1.7)--(5,-3.4)--(6,-1.7)--(4,-1.7);
\draw[directed] (5,3.4)--(6,5.1);
\draw[directed] (5,-3.4)--(6,-1.7);
\draw[directed] (6.6,4.08)--(6.5,4.25);
\draw[directed] (6.5,4.25)--(6.4,4.42);
\draw[directed] (0.6,0.68)--(0.5,0.85);
\draw[directed] (0.5,0.85)--(0.4,1.02);
\draw (0,1.7) node[above] {7};
\draw (2,1.7) node[above] {5};
\draw (4,1.7) node[above left] {2};
\draw (5,3.4) node[above left] {3};
\draw (6,5.1) node[above] {7};
\draw (7,3.4) node[right] {4};
\draw (6,1.7) node[below right] {$v$};
\draw (5,0) node[right] {1};
\draw (6,-1.7) node[right] {7};
\draw (5,-3.4) node[below] {3};
\draw (4,-1.7) node[below left] {8};
\draw (3,0) node[below left] {6};
\draw (1,0) node[below] {4};
\filldraw[black] (0,1.7) circle (2pt);
\filldraw[black] (2,1.7) circle (2pt);
\filldraw[black] (4,1.7) circle (2pt);
\filldraw[black] (5,3.4) circle (2pt);
\filldraw[black] (6,5.1) circle (2pt);
\filldraw[black] (7,3.4) circle (2pt);
\filldraw[black] (6,1.7) circle (2pt);
\filldraw[black] (5,0) circle (2pt);
\filldraw[black] (6,-1.7) circle (2pt);
\filldraw[black] (5,-3.4) circle (2pt);
\filldraw[black] (4,-1.7) circle (2pt);
\filldraw[black] (3,0) circle (2pt);
\filldraw[black] (1,0) circle (2pt);
\draw (3,-5) node[below] {$M(G)$};
\end{tikzpicture}
\end{center}
\caption{Graph and Matching Complex for Case III.1.b\label{III1b}}
\end{figure}

\textbf{Case III.2}.
Assume vertices 2 and 3 are interior vertices of $M$.  Then
$G_{\overline{2}}\in\{2P_3, C_5, K_{3,2}\}$.  We consider each of those
cases.

\textbf{Case III.2.a}. If $G_{\overline{2}}=2P_3$, then $G_{\overline{2}}$ consists
of the path with edges $\overline{1}$ and $\overline{3}$ and another $P_3$ path
with edges $\overline{v}$ and $\overline{5}$, for some edge $\overline{5}$ containing a
vertex of $\overline{v}$ and another vertex, not in the edges $\overline{1}$,
$\overline{2}$, $\overline{3}$, $\overline{4}$.  All other edges of $G$ must be incident
to  $\overline{2}$ and $\overline{v}$.  Then $G_{\overline{4}}$ is a disconnected subgraph
of $G$, with the edge $\overline{3}$ forming one component.  However, 4 is a
boundary vertex of $M$, so this is not possible (in case III).

\textbf{Case III.2.b}. If $G_{\overline{2}}=C_5$, then without loss of generality
$G$ contains the edges $\overline{5}=\{x_1,x_v\}$ and $\overline{6}=\{y_3,y_v\}$
(so $G_{\overline{2}}$ is $C_5: x_1,y_1,y_3,y_v,x_v,x_1$), and all other edges
of $G$ contain a vertex of $\overline{v}$ and a vertex of $\overline{2}$.  Then
$G_{\overline{4}}$, which contains $P_4: y_1,y_3,y_v,x_v$, must be $P_5$,
ending in edge $\overline{7}=\{x_v,x_2\}$.  Now $G_{\overline{5}}$ contains
$P_3\sqcup P_2$, with vertices $y_1, y_3, y_v; x_2, y_2$.
If $G_{\overline{5}}$ is just $P_3\sqcup P_2$, then it is covered in Case II.
Otherwise, there is one more edge $\overline{8} = \{y_v, y_2\}$, making
$G_{\overline{5}}= P_5$.  Any other edge would form a triangle in $G$, so
we have described all edges of $G$; $G$ and $M$ are shown in
Figure~\ref{III2b}.
This is part~3(c) of the Theorem.

\begin{figure}[h]
\begin{center}

\begin{tikzpicture}
\draw (0,0)--(1,0) node[draw=none,fill=none,font=\footnotesize,midway,below] {$\overline{4}$};
\draw (1,0)--(2,0) node[draw=none,fill=none,font=\footnotesize,midway,below] {$\overline{1}$};
\draw (2,0)--(3,1) node[draw=none,fill=none,font=\footnotesize,midway,right] {$\overline{3}$};
\draw (3,1)--(2,2) node[draw=none,fill=none,font=\footnotesize,midway,right] {$\overline{6}$};
\draw (2,2)--(1,2) node[draw=none,fill=none,font=\footnotesize,midway,above] {$\overline{v}$};
\draw (1,2)--(0,2) node[draw=none,fill=none,font=\footnotesize,midway,above] {$\overline{7}$};
\draw (0,2)--(0,0) node[draw=none,fill=none,font=\footnotesize,midway,left] {$\overline{2}$};
\draw (1,1)--(2,2) node[draw=none,fill=none,font=\footnotesize,midway,below] {$\overline{8}$};
\draw (1,0)--(1,1) node[draw=none,fill=none,font=\footnotesize,midway,right] {$\overline{5}$};
\draw (0,0)--(1,1);
\draw (1,1)--(1,2);
\filldraw[black] (0,0) circle (2pt);
\filldraw[black] (1,0) circle (2pt);
\filldraw[black] (2,0) circle (2pt);
\filldraw[black] (0,2) circle (2pt);
\filldraw[black] (1,2) circle (2pt);
\filldraw[black] (2,2) circle (2pt);
\filldraw[black] (3,1) circle (2pt);
\draw (1,-1) node[below] {$G$};
\end{tikzpicture}
\hspace*{12pt}
\begin{tikzpicture}[scale=.7]
\shadedraw (1,0)--(2,2)--(0,2)--(1,0);
\shadedraw (1,0)--(3,0)--(2,2)--(1,0);
\shadedraw (3,0)--(4,2)--(2,2)--(3,0);
\shadedraw (3,0)--(5,0)--(4,2)--(3,0);
\shadedraw (5,0)--(6,2)--(4,2)--(5,0);
\shadedraw (5,0)--(7,0)--(6,2)--(5,0);
\shadedraw (7,0)--(8,2)--(6,2)--(7,0);
\shadedraw (0,2)--(2,2)--(3,3.5)--(0,2);
\shadedraw (2,2)--(4,2)--(3,3.5)--(2,2);
\shadedraw (4,2)--(5,3.5)--(3,3.5)--(4,2);
\shadedraw (4,2)--(6,2)--(5,3.5)--(4,2);
\shadedraw (6,2)--(8,2)--(5,3.5)--(6,2);
\filldraw[black] (1,0) circle (2pt);
\filldraw[black] (3,0) circle (2pt);
\filldraw[black] (5,0) circle (2pt);
\filldraw[black] (7,0) circle (2pt);
\filldraw[black] (0,2) circle (2pt);
\filldraw[black] (2,2) circle (2pt);
\filldraw[black] (4,2) circle (2pt);
\filldraw[black] (6,2) circle (2pt);
\filldraw[black] (8,2) circle (2pt);
\filldraw[black] (3,3.5) circle (2pt);
\filldraw[black] (5,3.5) circle (2pt);
\draw (1,0) node[below left] {6};
\draw (3,0) node[below] {4};
\draw (5,0) node[below] {$v$};
\draw (7,0) node[below right] {1};
\draw (0,2) node[left] {1};
\draw (2,2) node[above left] {7};
\draw (4,2) node[above] {3};
\draw (6,2) node[above right] {2};
\draw (8,2) node[right] {6};
\draw (3,3.5) node[above] {8};
\draw (5,3.5) node[above] {5};
\draw[directed] (0,2)--(1,0);
\draw[directed] (7,0)--(8,2);
\draw (4,-1) node[below] {$M(G)$};
\end{tikzpicture}
\caption{Graph and Matching Complex for Case III.2.b\label{III2b}}
\end{center}
\end{figure}
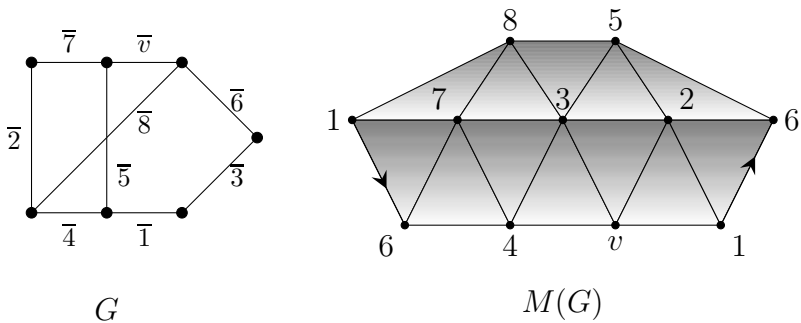

\textbf{Case III.2.c}.  If $G_{\overline{2}}=K_{3,2}$, then without loss of
generality $G$ contains the three edges $\overline{5}=\{x_v,y_1\}$,
$\overline{6}=\{y_v,y_3\}$ and $\overline{7}=\{y_v,x_1\}$.  Then $G_{\overline{4}}$
contains $C_4: y_1,x_v,y_v,y_3,y_1$, and so must be $\Gamma$.  To complete
$\Gamma$, there must be one more edge from $x_2$ to $\overline{v}$.
If that edge is $\{x_2,x_v\}$, then $G_{\overline{3}}$ is $C_5$, and that case
is covered in Case III.2.b.  So let $\overline{8}=\{x_2,y_v\}$.  Then $G_{\overline{1}}$
contains three edges containing vertex $y_v$, so must be $\Gamma$.
Therefore $G$ also contains the edge $\overline{9}=\{y_2,x_v\}$.  Any other
edge would create a triangle or would disrupt $G_{\overline{1}}$ or $G_{\overline{4}}$,
so we have described all edges of $G$; $G$ and $M(G)$ are shown in
Figure~\ref{III2c}.
This is part~4(b) of the Theorem.

\begin{figure}[h]
\begin{tikzpicture}
\draw (0,0)--(2,0)--(4,0)--(3,1)--(2,2)--(1,1)--(0,0);
\draw (0,0)--(2,4)--(4,0);
\draw (2,0)--(2,2)--(2,4);
\filldraw[black] (0,0) circle (2pt);
\filldraw[black] (2,0) circle (2pt);
\filldraw[black] (4,0) circle (2pt);
\filldraw[black] (1,1) circle (2pt);
\filldraw[black] (3,1) circle (2pt);
\filldraw[black] (2,2) circle (2pt);
\filldraw[black] (2,4) circle (2pt);
\draw (0,0)--(2,0) node[draw=none,fill=none,font=\footnotesize,midway,below] {$\overline{4}$};
\draw (2,0)--(4,0) node[draw=none,fill=none,font=\footnotesize,midway,below] {$\overline{1}$};
\draw (0,0)--(1,1) node[draw=none,fill=none,font=\footnotesize,midway,right] {$\overline{2}$};
\draw (1,1)--(2,2) node[draw=none,fill=none,font=\footnotesize,midway,left] {$\overline{8}$};
\draw (2,0)--(2,2) node[draw=none,fill=none,font=\footnotesize,midway,right] {$\overline{7}$};
\draw (4,0)--(3,1) node[draw=none,fill=none,font=\footnotesize,midway,left] {$\overline{3}$};
\draw (3,1)--(2,2) node[draw=none,fill=none,font=\footnotesize,midway,right] {$\overline{6}$};
\draw (0,0)--(2,4) node[draw=none,fill=none,font=\footnotesize,midway,left] {$\overline{9}$};
\draw (4,0)--(2,4) node[draw=none,fill=none,font=\footnotesize,midway,right] {$\overline{5}$};
\draw (2,2)--(2,4) node[draw=none,fill=none,font=\footnotesize,midway,right] {$\overline{v}$};
\draw (2,-1) node[below] {$G$};
\end{tikzpicture}
\hspace{18pt}
\begin{tikzpicture}
\shadedraw (0,1.3)--(2,1.3)--(1,3)--(0,1.3);
\shadedraw (2,1.3)--(3,3)--(1,3)--(2,1.3);
\shadedraw (2,1.3)--(4,1.3)--(3,3)--(2,1.3);
\shadedraw (4,1.3)--(5,3)--(3,3)--(4,1.3);
\shadedraw (4,1.3)--(6,1.3)--(5,3)--(4,1.3);
\shadedraw (6,1.3)--(7,3)--(5,3)--(6,1.3);
\shadedraw (6,1.3)--(8,1.3)--(7,3)--(6,1.3);
\shadedraw (1,3)--(2,4.7)--(0,4.7)--(1,3);
\shadedraw (1,3)--(3,3)--(2,4.7)--(1,3);
\shadedraw (3,3)--(4,4.7)--(2,4.7)--(3,3);
\shadedraw (3,3)--(5,3)--(4,4.7)--(3,3);
\shadedraw (5,3)--(6,4.7)--(4,4.7)--(5,3);
\shadedraw (5,3)--(7,3)--(6,4.7)--(5,3);
\shadedraw (7,3)--(8,4.7)--(6,4.7)--(7,3);
\draw[directed] (2,1.3)--(0,1.3);
\draw[directed] (8,4.7)--(6,4.7);
\draw[directed] (6.8,1.3)--(7,1.3);
\draw[directed] (7,1.3)--(7.2,1.3);
\draw[directed] (0.8,4.7)--(1,4.7);
\draw[directed] (1,4.7)--(1.2,4.7);
\draw (0,1.3) node[below left] {9};
\draw (2,1.3) node[below] {1};
\draw (4,1.3) node[below] {$v$};
\draw (6,1.3) node[below] {4};
\draw (8,1.3) node[below right] {5};
\draw (1,3) node[left] {6};
\draw (3,3) node[above] {2};
\draw (5,3) node[above] {3};
\draw (7,3) node[right] {8};
\draw (0,4.7) node[above left] {4};
\draw (2,4.7) node[above] {5};
\draw (4,4.7) node[above] {7};
\draw (6,4.7) node[above] {9};
\draw (8,4.7) node[above right] {1};
\filldraw[black] (0,1.3) circle (2pt);
\filldraw[black] (2,1.3) circle (2pt);
\filldraw[black] (4,1.3) circle (2pt);
\filldraw[black] (6,1.3) circle (2pt);
\filldraw[black] (8,1.3) circle (2pt);
\filldraw[black] (1,3) circle (2pt);
\filldraw[black] (3,3) circle (2pt);
\filldraw[black] (5,3) circle (2pt);
\filldraw[black] (7,3) circle (2pt);
\filldraw[black] (0,4.7) circle (2pt);
\filldraw[black] (2,4.7) circle (2pt);
\filldraw[black] (4,4.7) circle (2pt);
\filldraw[black] (6,4.7) circle (2pt);
\filldraw[black] (8,4.7) circle (2pt);
\draw (4,.25) node[below] {$M(G)$};
\end{tikzpicture}

\caption{Graph and Matching Complex for Case III.2.c\label{III2c}}
\end{figure}
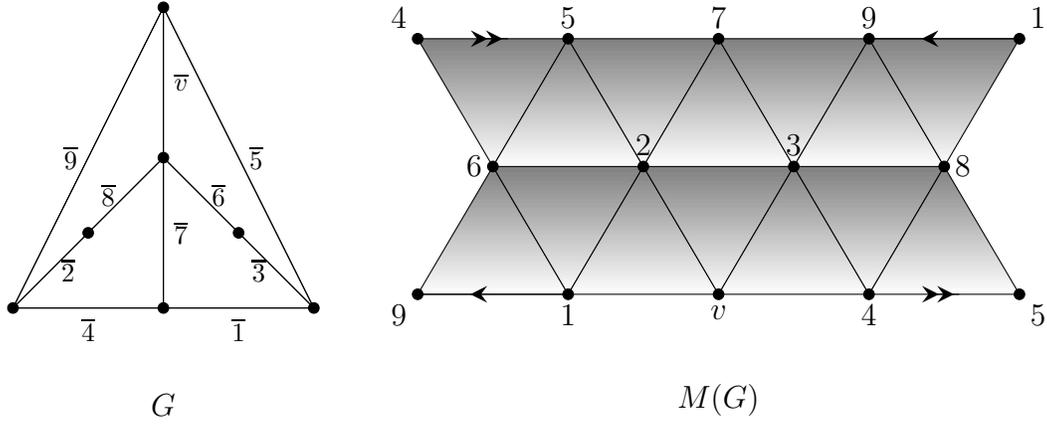

\textbf{Case IV}.
All boundary vertices $v$ of $M$ have $\mbox{link}_{M}(v)=P_5$, so
$G_{\overline{v}}=\Gamma$.
Fix such a boundary vertex $v$, and let the link of $v$ have vertices
1, 2, 3, 4, 5 (forming a path in that order).  Then 1 and 5 are boundary
vertices of $M$, and $G_{\overline{1}}$ and $G_{\overline{5}}$ are both $\Gamma$.
See Figure~\ref{IV} for the subgraph of $G$ containing $\overline{v}$ and
$G_{\overline{v}}$.
Every other edge of $G$ must contain a vertex of $\overline{v}$.

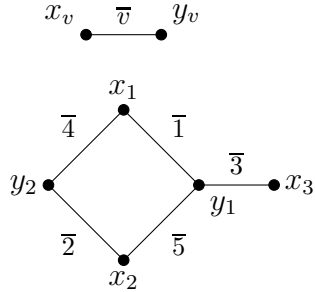
\begin{figure}[h]
\begin{center}
\begin{tikzpicture}
\draw (.5,3)--(1.5,3) node[draw=none,fill=none,font=\footnotesize,midway,above] {$\overline{v}$};
\draw (0,1)--(1,2) node[draw=none,fill=none,font=\footnotesize,midway,above left] {$\overline{4}$};
\draw (1,2)--(2,1) node[draw=none,fill=none,font=\footnotesize,midway,above right] {$\overline{1}$};
\draw (2,1)--(3,1) node[draw=none,fill=none,font=\footnotesize,midway,above] {$\overline{3}$};
\draw (2,1)--(1,0) node[draw=none,fill=none,font=\footnotesize,midway,below right] {$\overline{5}$};
\draw (1,0)--(0,1) node[draw=none,fill=none,font=\footnotesize,midway,below left] {$\overline{2}$};
\filldraw[black] (0,1) circle (2pt);
\filldraw[black] (2,1) circle (2pt);
\filldraw[black] (3,1) circle (2pt);
\filldraw[black] (1,0) circle (2pt);
\filldraw[black] (1,2) circle (2pt);
\filldraw[black] (.5,3) circle (2pt);
\filldraw[black] (1.5,3) circle (2pt);
\draw (0,1) node[left] {$y_2$};
\draw (2,1) node[below right] {$y_1$};
\draw (3,1) node[right] {$x_3$};
\draw (1,0) node[below] {$x_2$};
\draw (1,2) node[above] {$x_1$};
\draw (.5,3) node[above left] {$x_v$};
\draw (1.5,3) node[above right] {$y_v$};
\end{tikzpicture}
\caption{Subgraph $G_{\overline{v}}$ for Case IV.\label{IV}}
\end{center}
\end{figure}

Now consider $G_{\overline{1}}$, which is also $\Gamma$ (since 1 is a
boundary vertex of $M$).  The graph $G_{\overline{1}}$ contains the edges
$\overline{2}$ and $\overline{v}$, but not the edges $\overline{3}$, $\overline{4}$, or $\overline{5}$.
To complete $G_{\overline{1}}$ to $\Gamma$, there must be three more edges,
all incident to $\overline{v}$.
Without loss of generality, two of those edges are
$\overline{6}=\{x_v,y_2\}$ and $\overline{7}= \{y_v,x_2\}$.
If the third edge is not incident to $\overline{3}$, then $G_{\overline{3}}$
contains $G_{\overline{1}}$ and the edge $\overline{4}$.  With six edges,
$G_{\overline{3}}$ must be $K_{3,2}$, with vertex partition (without loss of
generality)
$\{y_2, y_v\}$, $\{x_v, x_1, x_2\}$.  Then $G_{\overline{5}}$ contains
the 4-cycle $x_v, y_2, x_1, y_v$, but this cannot be completed to a $\Gamma$
using only edges incident to $\overline{v}$ and $\overline{1}$.
Thus, one of the edges of $G_{\overline{1}}$ is incident to $\overline{3}$.
There are two possibilities, depending on its vertex in
edge $\overline{v}$.  See Figure~\ref{IVa}, where the vertices and edges are
placed differently to illustrate the constructions to follow.

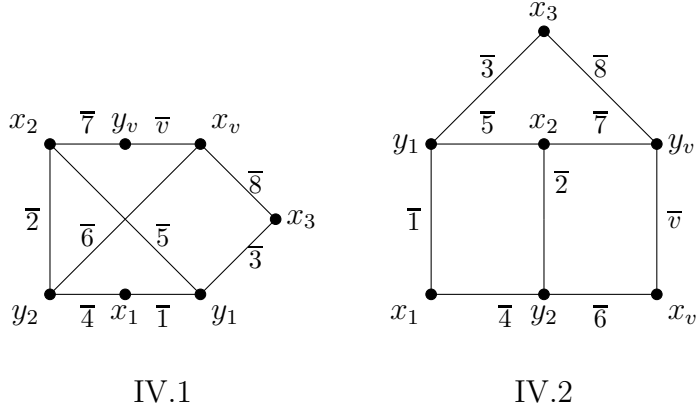
\begin{figure}
\begin{center}
\begin{tikzpicture}
\filldraw[black] (0,0) circle (2pt);
\filldraw[black] (1,0) circle (2pt);
\filldraw[black] (2,0) circle (2pt);
\filldraw[black] (3,1) circle (2pt);
\filldraw[black] (0,2) circle (2pt);
\filldraw[black] (1,2) circle (2pt);
\filldraw[black] (2,2) circle (2pt);
\draw (0,0)--(1,0) node[draw=none,fill=none,font=\footnotesize,midway,below] {$\overline{4}$};
\draw (1,0)--(2,0) node[draw=none,fill=none,font=\footnotesize,midway,below] {$\overline{1}$};
\draw (2,0)--(3,1) node[draw=none,fill=none,font=\footnotesize,midway,right] {$\overline{3}$};
\draw (3,1)--(2,2) node[draw=none,fill=none,font=\footnotesize,midway,right] {$\overline{8}$};
\draw (1,2)--(2,2) node[draw=none,fill=none,font=\footnotesize,midway,above] {$\overline{v}$};
\draw (0,2)--(1,2) node[draw=none,fill=none,font=\footnotesize,midway,above] {$\overline{7}$};
\draw (0,0)--(0,2) node[draw=none,fill=none,font=\footnotesize,midway,left] {$\overline{2}$};
\draw (0,0)--(1,1) node[draw=none,fill=none,font=\footnotesize,midway,above] {$\overline{6}$};
\draw (1,1)--(2,2);
\draw (0,2)--(1,1);
\draw (1,1)--(2,0) node[draw=none,fill=none,font=\footnotesize,midway,above] {$\overline{5}$};
\draw (0,0) node[below left] {$y_2$};
\draw (1,0) node[below] {$x_1$};
\draw (2,0) node[below right] {$y_1$};
\draw (0,2) node[above left] {$x_2$};
\draw (1,2) node[above] {$y_v$};
\draw (2,2) node[above right] {$x_v$};
\draw (3,1) node[right] {$x_3$};
\draw (1.5, -1) node[below] {IV.1};
\end{tikzpicture}
\hspace*{12pt}
\begin{tikzpicture}
\filldraw[black] (0,0) circle (2pt);
\filldraw[black] (1.5,0) circle (2pt);
\filldraw[black] (3,0) circle (2pt);
\filldraw[black] (0,2) circle (2pt);
\filldraw[black] (1.5,2) circle (2pt);
\filldraw[black] (3,2) circle (2pt);
\filldraw[black] (1.5,3.5) circle (2pt);
\draw (0,2)--(1.5,2) node[draw=none,fill=none,font=\footnotesize,midway,above] {$\overline{5}$};
\draw (0,0)--(1.5,0) node[draw=none,fill=none,font=\footnotesize,midway,below right] {$\overline{4}$};
\draw (1.5,0)--(3,0) node[draw=none,fill=none,font=\footnotesize,midway,below] {$\overline{6}$};
\draw (3,0)--(3,2) node[draw=none,fill=none,font=\footnotesize,midway,right] {$\overline{v}$};
\draw (1.5,2)--(3,2) node[draw=none,fill=none,font=\footnotesize,midway,above] {$\overline{7}$};
\draw (1.5,3.5)--(3,2) node[draw=none,fill=none,font=\footnotesize,midway,above] {$\overline{8}$};
\draw (0,0)--(0,2) node[draw=none,fill=none,font=\footnotesize,midway,left] {$\overline{1}$};
\draw (0,2)--(1.5,3.5) node[draw=none,fill=none,font=\footnotesize,midway,above] {$\overline{3}$};
\draw (1.5,1)--(1.5,2) node[draw=none,fill=none,font=\footnotesize,midway,right] {$\overline{2}$};
\draw (1.5,0)--(1.5,1);
\draw (0,0) node[below left] {$x_1$};
\draw (1.5,0) node[below] {$y_2$};
\draw (3,0) node[below right] {$x_v$};
\draw (0,2) node[left] {$y_1$};
\draw (1.5,2) node[above] {$x_2$};
\draw (3,2) node[right] {$y_v$};
\draw (1.5,3.5) node[above] {$x_3$};
\draw (1.5, -1) node[below] {IV.2};
\end{tikzpicture}
\caption{Two possible subgraphs of $G$ containing $G_{\overline{v}}$ and
$G_{\overline{1}}$ for Case IV.\label{IVa}}
\end{center}
\end{figure}

\textbf{Case IV.1}
Assume the edge connecting $x_3$ to $\overline{v}$ is $\overline{8}=\{x_3,x_v\}$.
Note that
any other edges in $G$ must be incident to edges $\overline{v}$ and $\overline{1}$.
Now consider $G_{\overline{5}}$, another $\Gamma$, since 5 is a boundary vertex of
$M$.
So far $G_{\overline{5}}$ contains the edges $\overline{4}$, $\overline{6}$, $\overline{v}$
and $\overline{8}$.  To complete it to $\Gamma$, the edge
$\overline{9}=\{x_1,y_v\}$ is needed.
Note that including any other edge in $G$ would create a triangle, which
is not allowed.  So it remains to find the matching complex of the graph IV.1 of
Figure~\ref{IVa} with the single edge $\overline{9}$ added.  The graph and its
matching complex (a M\"{o}bius strip) are shown in Figure~\ref{IV1}.
This is part~3(d) of the theorem.

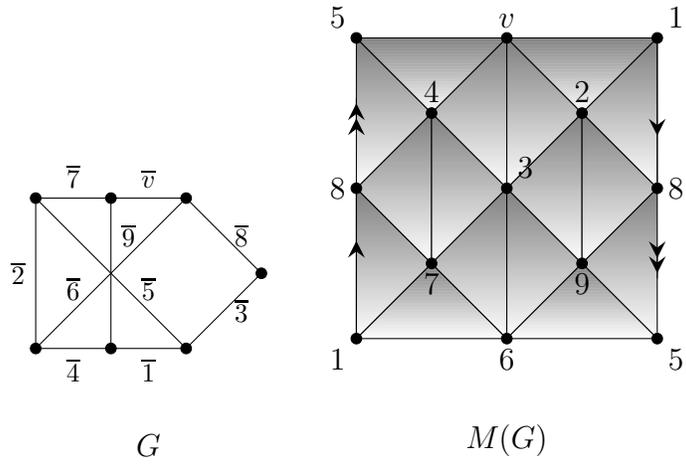
\begin{figure}[p]
\begin{center}

\begin{tikzpicture}
\filldraw[black] (0,0) circle (2pt);
\filldraw[black] (1,0) circle (2pt);
\filldraw[black] (2,0) circle (2pt);
\filldraw[black] (3,1) circle (2pt);
\filldraw[black] (0,2) circle (2pt);
\filldraw[black] (1,2) circle (2pt);
\filldraw[black] (2,2) circle (2pt);
\draw (0,0)--(1,0) node[draw=none,fill=none,font=\footnotesize,midway,below] {$\overline{4}$};
\draw (1,0)--(2,0) node[draw=none,fill=none,font=\footnotesize,midway,below] {$\overline{1}$};
\draw (2,0)--(3,1) node[draw=none,fill=none,font=\footnotesize,midway,right] {$\overline{3}$};
\draw (3,1)--(2,2) node[draw=none,fill=none,font=\footnotesize,midway,right] {$\overline{8}$};
\draw (1,2)--(2,2) node[draw=none,fill=none,font=\footnotesize,midway,above] {$\overline{v}$};
\draw (0,2)--(1,2) node[draw=none,fill=none,font=\footnotesize,midway,above] {$\overline{7}$};
\draw (0,0)--(0,2) node[draw=none,fill=none,font=\footnotesize,midway,left] {$\overline{2}$};
\draw (0,0)--(1,1) node[draw=none,fill=none,font=\footnotesize,midway,above] {$\overline{6}$};
\draw (1,1)--(2,2);
\draw (0,2)--(1,1);
\draw (1,1)--(2,0) node[draw=none,fill=none,font=\footnotesize,midway,above] {$\overline{5}$};
\draw (1,0)--(1,1);
\draw (1,1)--(1,2) node[draw=none,fill=none,font=\footnotesize,midway,right] {$\overline{9}$};
\draw (1.5, -1) node[below] {$G$};
\end{tikzpicture}
\hspace*{12pt}
\begin{tikzpicture}
\shadedraw (0,0)--(2,0)--(1,1)--(0,0);
\shadedraw (0,0)--(1,1)--(0,2)--(0,0);
\shadedraw (0,2)--(1,1)--(1,3)--(0,2);
\shadedraw (0,2)--(1,3)--(0,4)--(0,2);
\shadedraw (0,4)--(1,3)--(2,4)--(0,4);
\shadedraw (1,3)--(2,2)--(2,4)--(1,3);
\shadedraw (1,1)--(2,2)--(1,3)--(1,1);
\shadedraw (2,0)--(2,2)--(1,1)--(2,0);
\shadedraw (2,0)--(3,1)--(2,2)--(2,0);
\shadedraw (2,2)--(3,1)--(3,3)--(2,2);
\shadedraw (2,2)--(3,3)--(2,4)--(2,2);
\shadedraw (2,4)--(3,3)--(4,4)--(2,4);
\shadedraw (3,3)--(4,2)--(4,4)--(3,3);
\shadedraw (3,1)--(4,2)--(3,3)--(3,1);
\shadedraw (3,1)--(4,0)--(4,2)--(3,1);
\shadedraw (2,0)--(4,0)--(3,1)--(2,0);
\draw[directed] (0,0)--(0,2);
\draw[directed] (4,4)--(4,2);
\draw[directed] (0,2.8)--(0,3);
\draw[directed] (0,3)--(0,3.2);
\draw[directed] (4,1.2)--(4,1);
\draw[directed] (4,1)--(4,.8);
\draw (0,0) node[below left] {1};
\draw (0,2) node[left] {8};
\draw (0,4) node[above left] {5};
\draw (1,1) node[below] {7};
\draw (1,3) node[above] {4};
\draw (2,0) node[below] {6};
\draw (2,2) node[above right] {3};
\draw (2,4) node[above] {$v$};
\draw (3,1) node[below] {9};
\draw (3,3) node[above] {2};
\draw (4,0) node[below right] {5};
\draw (4,2) node[right] {8};
\draw (4,4) node[above right] {1};
\filldraw[black] (0,0) circle (2pt);
\filldraw[black] (0,2) circle (2pt);
\filldraw[black] (0,4) circle (2pt);
\filldraw[black] (1,1) circle (2pt);
\filldraw[black] (1,3) circle (2pt);
\filldraw[black] (2,0) circle (2pt);
\filldraw[black] (2,2) circle (2pt);
\filldraw[black] (2,4) circle (2pt);
\filldraw[black] (3,1) circle (2pt);
\filldraw[black] (3,3) circle (2pt);
\filldraw[black] (4,0) circle (2pt);
\filldraw[black] (4,2) circle (2pt);
\filldraw[black] (4,4) circle (2pt);
\draw (2, -1) node[below] {$M(G)$};
\end{tikzpicture}
\caption{Graph and Manifold for Case IV.1\label{IV1}}
\end{center}
\end{figure}

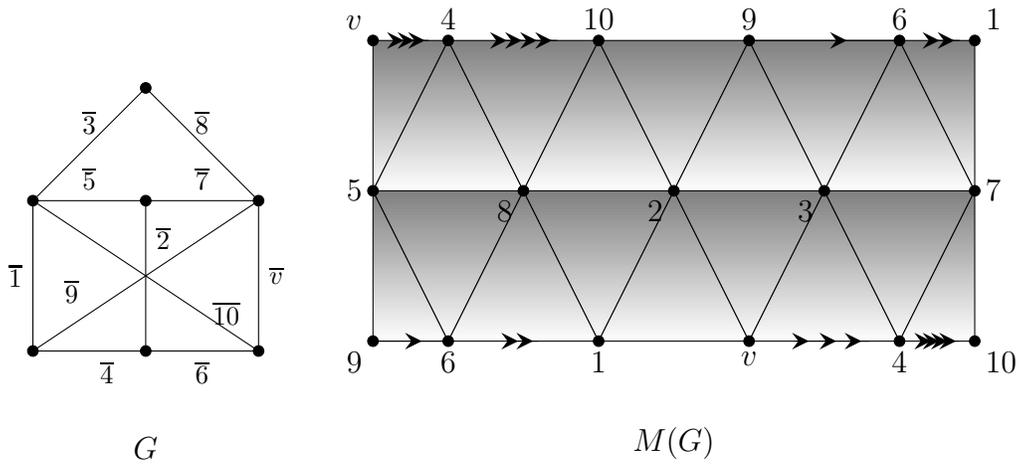
\begin{figure}
\begin{center}

\begin{tikzpicture}
\filldraw[black] (0,0) circle (2pt);
\filldraw[black] (1.5,0) circle (2pt);
\filldraw[black] (3,0) circle (2pt);
\filldraw[black] (0,2) circle (2pt);
\filldraw[black] (1.5,2) circle (2pt);
\filldraw[black] (3,2) circle (2pt);
\filldraw[black] (1.5,3.5) circle (2pt);
\draw (0,2)--(1.5,2) node[draw=none,fill=none,font=\footnotesize,midway,above] {$\overline{5}$};
\draw (0,0)--(1.5,0) node[draw=none,fill=none,font=\footnotesize,midway,below right] {$\overline{4}$};
\draw (1.5,0)--(3,0) node[draw=none,fill=none,font=\footnotesize,midway,below] {$\overline{6}$};
\draw (3,0)--(3,2) node[draw=none,fill=none,font=\footnotesize,midway,right] {$\overline{v}$};
\draw (0,0)--(1.5,1) node[draw=none,fill=none,font=\footnotesize,midway,above left] {$\overline{9}$};
\draw (1.5,2)--(3,2) node[draw=none,fill=none,font=\footnotesize,midway,above] {$\overline{7}$};
\draw (1.5,3.5)--(3,2) node[draw=none,fill=none,font=\footnotesize,midway,above] {$\overline{8}$};
\draw (0,0)--(0,2) node[draw=none,fill=none,font=\footnotesize,midway,left] {$\overline{1}$};
\draw (0,2)--(1.5,3.5) node[draw=none,fill=none,font=\footnotesize,midway,above] {$\overline{3}$};
\draw (1.5,1)--(1.5,2) node[draw=none,fill=none,font=\footnotesize,midway,right] {$\overline{2}$};
\draw (1.5,1)--(3,0) node[draw=none,fill=none,font=\footnotesize,midway,right] {$\overline{10}$};
\draw (1.5,1)--(3,2);
\draw (1.5,0)--(1.5,1);
\draw (0,2)--(1.5,1);
\draw (1.5, -1) node[below] {$G$};
\end{tikzpicture}
\hspace*{8pt}
\begin{tikzpicture}
\shadedraw (0,0)--(1,0)--(0,2)--(0,0);
\shadedraw (1,0)--(2,2)--(0,2)--(1,0);
\shadedraw (1,0)--(3,0)--(2,2)--(1,0);
\shadedraw (3,0)--(4,2)--(2,2)--(3,0);
\shadedraw (3,0)--(5,0)--(4,2)--(3,0);
\shadedraw (5,0)--(6,2)--(4,2)--(5,0);
\shadedraw (5,0)--(7,0)--(6,2)--(5,0);
\shadedraw (7,0)--(8,2)--(6,2)--(7,0);
\shadedraw (7,0)--(8,0)--(8,2)--(7,0);
\shadedraw (0,2)--(1,4)--(0,4)--(0,2);
\shadedraw (0,2)--(2,2)--(1,4)--(0,2);
\shadedraw (2,2)--(3,4)--(1,4)--(2,2);
\shadedraw (2,2)--(4,2)--(3,4)--(2,2);
\shadedraw (4,2)--(5,4)--(3,4)--(4,2);
\shadedraw (4,2)--(6,2)--(5,4)--(4,2);
\shadedraw (6,2)--(7,4)--(5,4)--(6,2);
\shadedraw (6,2)--(8,2)--(7,4)--(6,2);
\shadedraw (8,2)--(8,4)--(7,4)--(8,2);
\draw[directed] (0,0)--(1,0);
\draw[directed] (5,4)--(7,4);
\draw[directed] (1.8,0)--(2,0);
\draw[directed] (2,0)--(2.2,0);
\draw[directed] (7.4,4)--(7.6,4);
\draw[directed] (7.6,4)--(7.8,4);
\draw[directed] (.32,4)--(.46,4);
\draw[directed] (.46,4)--(.60,4);
\draw[directed] (.60,4)--(.74,4);
\draw[directed] (5.6,0)--(5.9,0);
\draw[directed] (5.9,0)--(6.3,0);
\draw[directed] (6.3,0)--(6.6,0);
\draw[directed] (1.65,4)--(1.85,4);
\draw[directed] (1.85,4)--(2.05,4);
\draw[directed] (2.05,4)--(2.25,4);
\draw[directed] (2.25,4)--(2.45,4);
\draw[directed] (7.35,0)--(7.46,0);
\draw[directed] (7.46,0)--(7.57,0);
\draw[directed] (7.57,0)--(7.68,0);
\draw[directed] (7.68,0)--(7.79,0);
\draw (0,0) node[below left] {9};
\draw (1,0) node[below] {6};
\draw (3,0) node[below] {1};
\draw (5,0) node[below] {$v$};
\draw (7,0) node[below] {4};
\draw (8,0) node[below right] {10};
\draw (0,2) node[left] {5};
\draw (2,2) node[below left] {8};
\draw (4,2) node[below left] {2};
\draw (6,2) node[below left] {3};
\draw (8,2) node[right] {7};
\draw (0,4) node[above left] {$v$};
\draw (1,4) node[above] {4};
\draw (3,4) node[above] {10};
\draw (5,4) node[above] {9};
\draw (7,4) node[above] {6};
\draw (8,4) node[above right] {1};
\filldraw[black] (0,0) circle (2pt);
\filldraw[black] (1,0) circle (2pt);
\filldraw[black] (3,0) circle (2pt);
\filldraw[black] (5,0) circle (2pt);
\filldraw[black] (7,0) circle (2pt);
\filldraw[black] (8,0) circle (2pt);
\filldraw[black] (0,2) circle (2pt);
\filldraw[black] (2,2) circle (2pt);
\filldraw[black] (4,2) circle (2pt);
\filldraw[black] (6,2) circle (2pt);
\filldraw[black] (8,2) circle (2pt);
\filldraw[black] (0,4) circle (2pt);
\filldraw[black] (1,4) circle (2pt);
\filldraw[black] (3,4) circle (2pt);
\filldraw[black] (5,4) circle (2pt);
\filldraw[black] (7,4) circle (2pt);
\filldraw[black] (8,4) circle (2pt);
\draw (4, -1) node[below] {$M(G)$};
\end{tikzpicture}
\caption{Graph and Manifold for Case IV.2\label{IV2}}
\end{center}
\end{figure}

\textbf{Case IV.2}
Assume the edge connecting $x_3$ to $\overline{v}$ is $\overline{8}=\{x_3,y_v\}$.
It is still the case that
any other edges in $G$ must be incident to edges $\overline{v}$ and $\overline{1}$.
Again consider $G_{\overline{5}}$, which is $\Gamma$.
This forces the edge $\overline{9}=\{x_1,y_v\}$.
At this point, any other edge of $G$ must be incident to $\overline{v}$,
$\overline{1}$ and $\overline{5}$.
Consider $G_{\overline{7}}$; it contains the edges $\overline{3}$, $\overline{1}$,
$\overline{4}$, and $\overline{6}$, which form a path $P_5$.
If $G_{\overline{7}}$ were just $P_5$, then $G$ and $M(G)$ would have been
considered in Case III.  So $G_{\overline{7}}$ must contain another
edge, and as it must be incident to $\overline{v}$, $\overline{1}$ and $\overline{5}$,
it must be the edge $\overline{10}=\{x_v,y_1\}$.  (This means that
$\overline{7}$ is a boundary vertex, with $G_{\overline{7}}=\Gamma$.)
Any additional edge would create a triangle.
So, we conclude that the graph for Case IV.2 must be that of Figure~\ref{IV2}.
The manifold is a torus with a $2$-ball removed.
This is part~4(c) of the Theorem.

This concludes the proof of Theorem~\ref{2d_with_boundary}.
\end{proof}

Just as in the case without boundary, while there are many $2$-dimensional manifold matching complexes with boundary, it turns out that if $d \ge 3$, then manifold matching complexes are very simple. The following theorem classifies all of them.

\begin{theorem}\label{ManifoldBoundaryThm}
Let $G$ be a simple graph such that $M(G)$ is a $d$-dimensional homology 
manifold with boundary for some $d\ge 3$.  
Then $M(G)$ is a combinatorial $d$-ball.  
In particular, $G$ is as described in Proposition~\ref{BallProp}.
\end{theorem}

\begin{proof}
Let $M=M(G)$ and $\dim M = d \geq 3$, and assume for lower dimensions that the only graphs which have homology spheres and balls as matching complexes are those graphs described in Propositions~\ref{SphereMatchingProp} and \ref{BallProp}, respectively. We will use this to show that $G$ is as described in Proposition~\ref{BallProp} and therefore $M$ is a combinatorial $d$-ball.

If $v\in M$ is an interior vertex then $\link_{M} v$ is a homology $(d-1)$-sphere, and if $v$ is a boundary vertex then $\link_{M}v$ is a homology $(d-1)$-ball. We already know from Theorems \ref{2dimMan} and \ref{ManifoldThm} that if $v$ is an interior vertex, then $G_{\overline{v}}$ is disconnected. We will focus on boundary vertices and analyze which sort of links can appear.

We will consider two cases:

\textbf{Case I}. For all boundary vertices $v \in M$, $G_{\overline{v}}$ is disconnected.

\textbf{Case II}. There exists a vertex $v$ in the boundary of  $M$ such that $G_{\overline{v}}$ is connected.

We will show that Case I implies that $G$ must be disconnected and that
$M$ must be a combinatorial ball. In Case II, we will show that either $G$ is disconnected or $G$ is the spider graph $\mbox{Sp}_{d+1}$. In either case, $M$ is a combinatorial ball.

\textbf{Case I}.
 For all vertices $v \in M$, $G_{\overline{v}}$ is disconnected.
In particular, for each boundary vertex $v\in M$, $G_{\overline{v}}$ contains a basic ball graph as a connected component.
We show that $G$ itself is disconnected.

\textbf{Case I.1}. There exists a boundary vertex $v \in M$ such that
$G_{\overline{v}}$ contains $P_2$ as a connected component.
Write $G_{\overline{v}}= P_2 \sqcup G'$, and
let $\overline{a}$ be the edge of $P_2$.
Consider $G_{\overline{a}}$, which is also disconnected.
Since all other edges of $G$ must contain a vertex of $\overline{v}$, all edges
of $G_{\overline{a}}$ are either incident to $\overline{v}$ or contained in $G'$.
Since $G_{\overline{a}}$ is disconnected part (or all) of $G'$ is not in the
same component of $G$ as $\overline{v}$.  Thus $G$ is not connected.

\textbf{Case I.2}.  There are no boundary vertices $v\in M$ such that
$G_{\overline{v}}$ contains $P_2$ as a connected component.
Let $v$ be a boundary vertex of $M$ and $G^o$ be a component of $G_{\overline{v}}$ that
is either $\Gamma$ or $\mbox{Sp}_k$ ($k\ge 2$), and
$G'=G_{\overline{v}}\setminus G^o$.  See Figures~\ref{ManBdThmCaseI.2MB}
and~\ref{ManBdThmCaseI.2detailMB}.

\begin{figure}[h]
\begin{center}
\begin{tikzpicture}[scale=1]
\node [cloud, draw,cloud puffs=10,cloud puff arc=120, aspect=2, inner ysep=1em] at (6,2.5) {$G'$}; 
\node [cloud, draw,cloud puffs=10,cloud puff arc=120, aspect=2, inner ysep=1em] at (1,2.5) {$G^o$}; 
\node[vertex] (f) at (3,5) {};
\node[vertex] (g) at (4,5) {};
\draw (f)--(g) node[draw=none,fill=none,font=\small,midway,above] {$\overline{v}$};
\end{tikzpicture}
\caption{The subgraph $G_{\overline{v}}$ for Case I.2  See
Figure~\ref{ManBdThmCaseI.2detailMB} for options for $G^o$.
\label{ManBdThmCaseI.2MB}}
\end{center}
\end{figure}
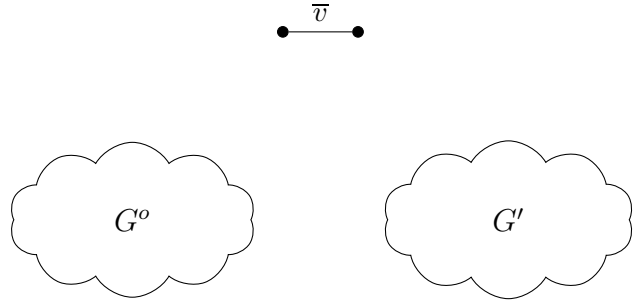

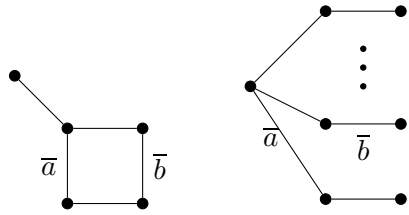
\begin{figure}[h]
\begin{center}
\begin{tikzpicture}
\node[vertex] (a) at (1,0) {};
\node[vertex] (b) at (2,0) {};
\node[vertex] (c) at (2,1) {};
\node[vertex] (d) at (1,1) {};
\node[vertex] (e) at (.3,1.7) {};
\draw (a)--(b) node[draw=none,fill=none,font=\small,very near start,above] {};
\draw (b)--(c) node[draw=none,fill=none,font=\small,midway,right] {$\overline{b}$};
\draw (c)--(d) node[draw=none,fill=none,font=\small,near end,below] {};
\draw (a)--(d) node[draw=none,fill=none,font=\small,midway,left] {$\overline{a}$};
\draw (d)--(e) node[draw=none,fill=none,font=\small,midway,above] {};
\end{tikzpicture}
\hspace*{18pt}
\begin{tikzpicture}
\node[vertex] (a) at (0,1.5) {};
\node[vertex] (b) at (1,0) {};
\node[vertex] (c) at (2,0) {};
\node[vertex] (d) at (1,1) {};
\node[vertex] (e) at (2,1) {};
\node[vertex] (h) at (1,2.5) {};
\node[vertex] (i) at (2,2.5) {};
\node[dot] (x) at (1.5,1.5) {};
\node[dot] (y) at (1.5,1.75) {};
\node[dot] (z) at (1.5,2) {};
\draw (a)--(b) node[draw=none,fill=none,font=\small,near start,below] {$\overline{a}$};
\draw (b)--(c) node[draw=none,fill=none,font=\small,midway,right] {};
\draw (a)--(d) node[draw=none,fill=none,font=\small,near end,below] {};
\draw (d)--(e) node[draw=none,fill=none,font=\small,midway,below] {$\overline{b}$};
\draw (a)--(h) node[draw=none,fill=none,font=\small,midway,above] {};
\draw (h)--(i) node[draw=none,fill=none,font=\small,midway,above] {};
\end{tikzpicture}
\caption{The options for $G^o$ for Case I.2
\label{ManBdThmCaseI.2detailMB}}
\end{center}
\end{figure}

For each possibility in Figure~\ref{ManBdThmCaseI.2detailMB} a special edge
$\overline{a}$ of $G^o$ and an edge $\overline{b}$ of $G^o$
not incident to $\overline{a}$ have been identified.
Since $G_{\overline{a}}$ is assumed not to contain $P_2$ as a connected
component, $\overline{b}$ must be connected to an edge of $\overline{v}$.
So $G^o$ is in the component of $G$ containing $\overline{v}$.
As in Case I.1,
since all other edges of $G$ must contain a vertex of $\overline{v}$, all edges
of $G_{\overline{a}}$ are either incident to $\overline{v}$ or contained in $G'$.
Since $G_{\overline{a}}$ is disconnected part (or all) of $G'$ is not in the
same component of $G$ as $\overline{v}$.  Thus $G$ is not connected.

Therefore in Case I, $G$ must be disconnected, and by Corollary \ref{BallCor} $M$ must be a  $d$-ball.

\textbf{Case II}. There exists a vertex $v$ in the boundary of $M$ such that $G_{\overline{v}}$ is connected.

By our induction hypothesis, Proposition~\ref{BallProp} lists all homology balls that arise as matching complexes for dimensions less than $d$. Since $\link_{M} v$ is a $(d-1)$-ball, this implies that $G_{\overline{v}} = \textrm{Sp}_{d}$. We consider $G_{\overline{a}_1}$ in Figure~\ref{ManBdThmCaseII} (ignore the dotted edge
$\overline{y}$).

\textbf{Case II.1}. The vertex $a_1$ is an interior vertex of $M$.

This implies that $G_{\overline{a}_1}$ is a disjoint union of copies of elements from $\mathcal{SG}$. Since the path $\overline{a}_2,\overline{b}_2,\overline{b}_d,\overline{a}_d$ is in $G_{\overline{a}_1}$, this path is part of either $C_5$ or $K_{3,2}$. But the only edges of $G$ not included in Figure~\ref{ManBdThmCaseII} must be incident to $\overline{v}$,
which is a contradiction. Therefore this subcase is impossible.

\textbf{Case II.2}. The vertex $a_1$ is a boundary vertex of $M$.

Again considering Figure~\ref{ManBdThmCaseII} (ignore the dotted edge $\overline{y}$),
we see that the subgraph $G_{\overline{a}_1}$ contains $\mbox{Sp}_{d-1}$ and $\overline{v}$. Since $\link_{M}a_1$ is a homology $(d-1)$-ball, our induction hypothesis says that there are only three options for the corresponding subgraph.

\textbf{Case II.2.a}. $G_{\overline{a}_1} = \mbox{Sp}_{d-1} \sqcup P_2$
or $G_{\overline{a}_1} = \mbox{Sp}_{d-1} \sqcup P_3$
The graph $G_{\overline{a}_1}$ is shown in
Figure~\ref{ManBdThmCaseII}, with a dotted edge $\overline{y}$ for the extra
edge in the case $\mbox{Sp}_{d-1} \sqcup P_3$.
The only other possible edges of $G$ connect one of the vertices of
$\overline{a}_1$ to one of the vertices of the edge $\overline{v}$.
Assume any one of these four edges exists and call it $\overline{x}$.

We now consider $\link_{M} \overline{a}_2$, which by induction must be as described in Proposition~\ref{BallProp}. However, we can see that $G_{\overline{a}_2}$ must contain an induced path of length $5$ (either $\overline{x},\overline{a}_1,\overline{b}_1,\overline{b}_d,\overline{a}_d$ or $\overline{v},\overline{x},\overline{b}_1,\overline{b}_d,\overline{a}_d$ depending on whether $\overline{x}$ is incident to $\overline{b}_1$ or not).
But none of the graphs
described in Proposition~\ref{BallProp} have an induced $P_6$.
Therefore no such edges $\overline{x}$ exist and so $G$ is exactly the graph pictured in Figure~\ref{ManBdThmCaseII}.

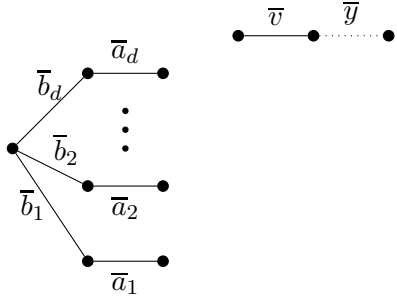
\begin{figure}[h]
\begin{center}
\begin{tikzpicture}[scale=1]

\node[vertex] (a) at (0,1.5) {};
\node[vertex] (b) at (1,0) {};
\node[vertex] (c) at (2,0) {};
\node[vertex] (d) at (1,1) {};
\node[vertex] (e) at (2,1) {};
\node[vertex] (f) at (3,3) {};
\node[vertex] (g) at (4,3) {};
\node[vertex] (h) at (1,2.5) {};
\node[vertex] (i) at (2,2.5) {};
\node[vertex] (j) at (5,3) {};
\node[dot] (x) at (1.5,1.5) {};
\node[dot] (y) at (1.5,1.75) {};
\node[dot] (z) at (1.5,2) {};
\draw (a)--(b) node[draw=none,fill=none,font=\small,near start,below] {$\overline{b}_1$};
\draw (b)--(c) node[draw=none,fill=none,font=\small,midway,below] {$\overline{a}_1$};
\draw (a)--(d) node[draw=none,fill=none,font=\small,near end,above] {$\overline{b}_2$};
\draw (d)--(e) node[draw=none,fill=none,font=\small,midway,below] {$\overline{a}_2$};
\draw (a)--(h) node[draw=none,fill=none,font=\small,midway,above] {$\overline{b}_d$};
\draw (h)--(i) node[draw=none,fill=none,font=\small,midway,above] {$\overline{a}_d$};
\draw (f)--(g) node[draw=none,fill=none,font=\small,midway,above] {$\overline{v}$};
\draw[dotted] (g)--(j) node[draw=none,fill=none,font=\small,midway,above] {$\overline{y}$};
\end{tikzpicture}
\caption{A subgraph for Case II.2.a. Any remaining edges of $G$ must be incident to both $\overline{v}$ and $\overline{a}_1$. \label{ManBdThmCaseII}}
\end{center}
\end{figure}

\textbf{Case II.2.b}. $G_{\overline{a}_1} = \mbox{Sp}_{d}$

In this case, $G$ contains the subgraph in Figure~\ref{ManBdThmCaseII.2.b}. Again, the only possible edges of $G$ that do not appear in Figure~\ref{ManBdThmCaseII.2.b} are edges that connect one of the vertices of $\overline{a}_1$ to one of the vertices of the edge $\overline{v}$. Again we consider $\link_{M} \overline{a}_2$. If any of these four possible edges exists, then $G_{\overline{a}_2}$ is $\mbox{Sp}_d$ with an additional edge connecting two of the legs of the spider. This again contradicts Proposition~\ref{BallProp}. Therefore $G$ is exactly the graph pictured in Figure~\ref{ManBdThmCaseII.2.b}.

\begin{figure}[h]
\begin{center}
\begin{tikzpicture}[scale=1]
\node[vertex] (a) at (0,1.5) {};
\node[vertex] (b) at (1,0) {};
\node[vertex] (c) at (2,0) {};
\node[vertex] (d) at (1,1) {};
\node[vertex] (e) at (2,1) {};
\node[vertex] (f) at (1,3.5) {};
\node[vertex] (g) at (2,3.5) {};
\node[vertex] (h) at (1,2.5) {};
\node[vertex] (i) at (2,2.5) {};
\node[dot] (x) at (1.5,1.5) {};
\node[dot] (y) at (1.5,1.75) {};
\node[dot] (z) at (1.5,2) {};
\draw (a)--(b) node[draw=none,fill=none,font=\small,near start,below] {$\overline{b}_1$};
\draw (b)--(c) node[draw=none,fill=none,font=\small,midway,below] {$\overline{a}_1$};
\draw (a)--(d) node[draw=none,fill=none,font=\small,near end,above] {$\overline{b}_2$};
\draw (d)--(e) node[draw=none,fill=none,font=\small,midway,below] {$\overline{a}_2$};
\draw (a)--(h) node[draw=none,fill=none,font=\small,very near end,below] {$\overline{b}_d$};
\draw (h)--(i) node[draw=none,fill=none,font=\small,midway,above] {$\overline{a}_d$};
\draw (f)--(g) node[draw=none,fill=none,font=\small,midway,above] {$\overline{v}$};
\draw (a)--(f) node[draw=none,fill=none,font=\small,midway,above] {$\overline{z}$};
\end{tikzpicture}
\caption{A subgraph for Case II.2.b. Any remaining edges of $G$ must be incident to both $\overline{v}$ and $\overline{a}_1$.\label{ManBdThmCaseII.2.b}}
\end{center}
\end{figure}
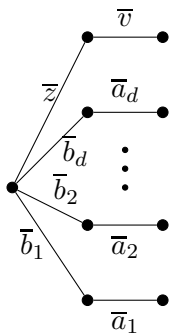

Therefore we have proved that either $G = \mbox{Sp}_{d+1}$ or that $G$ is disconnected. In the case that $G$ is disconnected, $M$ is a nontrivial join of some complexes $M(G_1)$ and $M(G_2)$ where $G = G_1 \sqcup G_2$. 
Therefore by Proposition~\ref{JoinManifold}, we know that $M$ is a homology
 $d$-ball. 
By our induction hypothesis, this means that $G_1$ and $G_2$ are exactly as 
described in Proposition~\ref{BallProp}, and therefore so is $G$, and $M(G)$
is a combinatorial $d$-ball. 
This completes the proof of the theorem.
\end{proof}

Thus we see that the homology manifolds with boundary that occur as
matching complexes are balls, the 2-dimensional annulus, the
M\"{o}bius strip, and the 2-dimensional torus with a ball removed.
Note that the balls that are matching complexes are all shellable,
as they are the joins of shellable balls and spheres.
Also, all of these matching complexes of dimension $d\ne 2$ have at most
$3d+3$ vertices, and those of dimension 2 have at most $11=3(2)+5$ vertices.

Jeli\'{c} Milutinovi\'{c}, et al.\ \cite{JJMV} show that (the 1-skeleton of)
a connected matching
complex has diameter at most 2 unless the graph has some pair of edges 
such that every other edge is incident to at least one edge of the pair.
 The only manifold matching complexes with 
diameter greater than 2 are $C_6 = M(K_{3,2})$, $P_4 = M(P_5)$ and 
$P_5 = M(\Gamma)$.  Note that the join of any two simplicial complexes 
has diameter at most 2.

\section{Further areas of research}

In this paper we have found all matching complexes that are homology
manifolds, both with and without boundary.  
We showed that all homology manifolds that arise as matching complexes are
in fact combinatorial manifolds.
Manifolds are, by definition, pure
complexes, meaning that all maximal faces are of the same size.  For matching complexes, this implies
that the corresponding graphs are ``equimatchable,'' that is, all maximal
matchings have the same size.
The question of which graphs are equimatchable in general is an ongoing area of research; see, for example, \cite{Frendrup}.

The {\em independence complex} of a graph is the simplicial complex whose
vertex set is the set of vertices of the graph with a face for each
independent (mutually nonadjacent) set of vertices.
The {\em line graph} of a graph $G$ is the graph $L(G)$ whose vertex set is
$E(G)$, and where two vertices of $L(G)$ are adjacent if the corresponding edges
of $G$ are incident.
The matching complex of a graph is the independence complex of its line
graph.  The class of all graphs is much larger than the class of line
graphs.  One could investigate which independence complexes of graphs
are homology manifolds.

Homology  manifolds are defined by conditions on links of faces. There are numerous well-studied properties of simplicial complexes that can be defined via similar link conditions, e.g., Buchsbaum, Cohen-Macaulay, and vertex decomposable complexes. Some of these have appeared in the context of matching complexes; for example, Ziegler showed in \cite{ziegler} that the $\nu_{m,n}$-skeleton of $M(K_{m,n})$ is vertex decomposable (and hence shellable), where $\nu_{m,n} = \min \left\lbrace m, \left\lfloor \frac{m+n+1}{3}\right\rfloor \right\rbrace -1.$ Similarly, Athanasiadis showed in \cite{Athan} that the $\nu_n$-skeleton of $M(K_n)$ is vertex decomposable, where $\nu_n = \left\lfloor\frac{n+1}{3}\right\rfloor -1$. (Shellability in this case was originally shown by Shareshian and Wachs  in \cite{SW07}.) 

We might ask the opposite sort of question: which complexes with various link condition properties can be realized as matching complexes? Our approach would be extended most naturally to Buchsbaum complexes. A pure $d$-dimensional complex is said to be \emph{Buchsbaum} if the link of any nonempty face $\sigma$ has the homology of a wedge of $(d-|\sigma|)$-spheres. All manifolds are Buchsbaum, but there are many matching complexes which are Buchsbaum but not manifolds. One simple example is $M(K_{1,3} \sqcup P_2)$.

Recall that matching complexes are flag complexes.
Frohmader~\cite{Frohmader} proved a conjecture of Kalai
(see \cite[Section III.4]{Stanley}) constructing, for every flag
complex $\Delta$, a balanced complex $\Delta'$ with the same number of faces 
of each dimension as $\Delta$.  
Chong and Nevo \cite{ChongNevo} extended this with a bound on the top homology.
(Balanced means the $1$-skeleton of the complex has
chromatic number equal to the size of the largest face in the complex.)
Many, but not all,  of our manifold matching complexes are balanced.
The basic sphere graph $C_5$ has matching complex $C_5$, which is clearly
not balanced.  Matching complexes that are spheres and balls arise from
graphs that are disjoint unions of basic sphere and ball graphs.  Any
such union that does not include $C_5$ has a balanced matching
complex.  The 2-dimensional torus, matching complex of $K_{4,3}$, is
balanced.  Among the matching complexes that are manifolds with boundary,
the triangulated M\"{o}bius strips are not balanced. (They all come from
graphs with odd cycles.)  But the others, the triangulated
annulus and the torus with 2-ball removed, are balanced.  For those that
are not balanced, it would be interesting to examine the corresponding
balanced complexes Frohmader's method would construct.
We could also ask what other matching complexes (that are not manifolds)
are balanced.

\section*{Acknowledgments}

This work was done as part of the 2018 Graduate Research Workshop in Combinatorics.  The workshop was partially funded by NSF  grants  1603823, 1604773, and  1604458,  ``Collaborative  Research:  Rocky  Mountain  -  Great Plains Graduate Research Workshops in Combinatorics'', NSA grant H98230-18-1-0017, ``The 2018 and 2019 Rocky Mountain - Great Plains Graduate Research Workshops in Combinatorics'', Simons Foundation Collaboration Grants \#316262 and \#426971, and grants from the Combinatorics Foundation and the Institute for Mathematics and its Applications.  Margaret Bayer also received support from the University of Kansas General 
Research Fund. Bennet Goeckner also received support from an AMS-Simons travel grant.
Marija Jeli\'{c} Milutinovi\'{c} also received support from
Grant \#174034 of the Ministry of Education, Science and Technological
Development of Serbia.

We would like to thank  Helen Jenne, Alex McDonough,
Jeremy Martin, George Nasr and Julianne Vega for their
contributions to the project during the workshop, as well as the
organizers of GRWC 2018.
We would also like to thank Isabella Novik and Russ Woodroofe for comments on earlier drafts of this paper.
Finally, we thank the referee for the careful reading and 
suggestions that improved the paper.


\vspace{12pt}

\textsc{Margaret Bayer}, University of Kansas, USA, bayer@ku.edu

\textsc{Bennet Goeckner}, University of Washington, USA, goeckner@uw.edu

\textsc{Marija Jeli\'{c} Milutinovi\'{c}}, University of Belgrade, Serbia, marijaj@matf.bg.ac.rs
\end{document}